\documentclass{amsart}
\usepackage{graphicx}  	
\usepackage{amsmath,amssymb} 
\usepackage{amsthm}
\usepackage{bm}  		
\usepackage{pdfsync}  	
\usepackage{subfigure}

\theoremstyle{plain}
\newtheorem{theorem}{Theorem}[section]
\newtheorem{proposition}[theorem]{Proposition}
\newtheorem{lemma}[theorem]{Lemma}
\newtheorem{corollary}[theorem]{Corollary}

\theoremstyle{definition}
\newtheorem{example}[theorem]{Example}
\newtheorem{definition}[theorem]{Definition}
\newtheorem{remark}[theorem]{Remark}

\newcommand{\C}{{\mathbb C}}
\newcommand{\R}{{\mathbb R}}
\newcommand{\N}{{\mathbb N}}

\newcommand{\X}{{\mathbb X}}
\newcommand{\Sc}{{\mathcal S}}
\newcommand{\U}{{\mathcal U}}
\newcommand{\V}{{\mathcal V}}
\newcommand{\x}{{\mathbf x}}

\newcommand{\aaa}{{\mathbf a}}
\newcommand{\alfa}{{\bm \alpha}}
\newcommand{\f}{{\mathbf f}}
\newcommand{\g}{{\mathbf g}}

\def\LT{\mathop{\rm LT}\nolimits}

\def\Mat{\mathop{\rm Mat}\nolimits}

\def\Spec{\mathop{\rm Spec}\nolimits}

\def\Jac{\mathop{\rm Jac}\nolimits}

\def\y1{\mathbf{y}^{(1)}}

\let\epsilon=\varepsilon
\let\rho=\varrho
\let\phi=\varphi
\let\To=\longrightarrow

\def\tfrac #1#2{{\textstyle\frac{#1}{#2}}}

\def\cocoa{\mbox{\rm
C\kern-.13em o\kern-.07 em C\kern-.13em o\kern-.15em A}}
\def\apcocoa{\mbox{\rm
A\kern-0.13em p\kern -0.07em C\kern-.13em o\kern-.07 em C\kern-.13em
o\kern-.15em A}}

\begin{document}

\title{Stable Complete Intersections}

\author{Lorenzo Robbiano}
\address{Dipartimento di Matematica, Universit\`a di Genova, Via
Dodecaneso 35,
I-16146 Genova, Italy}
\email{robbiano@dima.unige.it}

\author{Maria Laura Torrente}
\address{Dipartimento di Matematica, Universit\`a di Genova, Via
Dodecaneso 35,
I-16146 Genova, Italy}
\email{torrente@dima.unige.it}

\date{\today}

\begin{abstract}
A complete intersection of $n$ polynomials in $n$ 
indeterminates has only a finite number of zeros. 
In this paper we address the following question:
how do the zeros change when the coefficients of the 
polynomials are perturbed? In the first part we 
show how to construct semi-algebraic sets in the 
parameter space over which all the complete intersection ideals
share the same number of isolated real zeros. In the second part 
we show how to modify the complete intersection and get a 
new one which generates the same ideal but whose real zeros 
are more stable with respect to perturbations of the coefficients.

\end{abstract}


\maketitle
\section{Introduction}
What is the defining (or vanishing) ideal of a finite set $\X$ of points 
in the affine space? 
The standard answer is that it is the set of all  the polynomials which vanish
at~$\X$.  And there are very efficient methods to compute it, based on 
Buchberger-M\"oller's algorithm (see for 
instance~\cite{ABKR}, \cite{AKR} and~\cite{BM}). 

However, the logical and computational environment changes 
completely when the coordinates of the points are perturbed by errors, 
a situation which is normal when dealing with real world problems.
In that case one has to use  approximation and to consider the question
of stability. Introductory material  about this topic can be found in 
the book~\cite{RA},  in particular in the
paper~\cite{KPR} and its bibliography. 

The methods used so far share 
the strategy of modifying the Buchberger-M\"oller Algorithm and compute
a Gr\"obner basis or a border basis of an ideal of polynomials which
{\it almost vanish}\/ at $\X$ (see for instance~\cite{F} and~\cite{FT}). 
A key remark is that, whatever algorithm is used, 
at a certain moment one has computed $n$ polynomials $f_1, \dots, f_n$ 
which generate a zero-dimensional ideal. 
Since the dimension has dropped from $n$ to zero, 
the $n$ polynomials form a complete intersection 
which almost vanishes at~$\X$.
Further steps in the algorithm will be used to eliminate spurious 
points and to produce a Gr\"obner or border basis.  

Now, a complete intersection of $n$ polynomials in $n$ 
indeterminates has only a finite number of zeros, and the main 
question is: how do the zeros change when the coefficients of the 
polynomials are perturbed? Can we devise a strategy to make 
the situation reasonably stable? In other words, can we change 
the generating polynomials so that the stability of their 
common zeros increases? 
It is well-known that for a linear system with $n$ equations and $n$ 
unknowns, the most stable situation occurs when the coefficient 
matrix is orthonormal. Is there an analogue to orthonormality when  
we deal with polynomial systems?

In numerical analysis the condition number of a problem
measures the sensitivity of the solution to small changes
in the input data, and so it
reveals how numerically well-conditioned the problem is.
There exist a huge body of results about condition
numbers for various numerical problems, for instance
the solution of a linear system, the problem of matrix inversion,
the least squares problem, and the computation of eigenvalues and
eigenvectors.

On the other hand, not very much is known about condition numbers
of polynomial systems. As a notable exception we mention the
paper~\cite{SS93} of Shub and Smale  who
treated the case of zero-dimensional homogeneous polynomial
systems; later on their result was extended by 
D\'egot (see~\cite{D01}) to the case of positive-dimensional 
homogeneous polynomial systems.

Tackling the above mentioned problem entails a 
preliminary analysis of the following question of algebraic nature. 
If we are given a zero-dimensional complete intersection of 
polynomials with simple zeros, how far can we perturb 
the coefficients so that the zeros remain smooth
and their number does not change? 
It is quite clear that smoothness and constancy 
of the number of zeros are essential if we want to 
consider the perturbation a good one.

Starting with the classical idea that a perturbed system is a 
member of a family of systems, we describe a good subset of 
the parameter space over which the members of the family 
share the property that their zero sets
have the same number of smooth real points. 
This is the content of Section~\ref{Families of Complete Intersections}
where we describe a free (see~Proposition~\ref{flatness}), and a smooth
(see~Theorem~\ref{jacobian}) locus in the parameter space. 
Then we provide a suitable algorithm to 
compute what we call an $I$-optimal subscheme of the parameter 
space (see Corollary~\ref{algo-optimal}): it is a subscheme 
over which the complete intersection schemes are smooth 
and have the same number of complex points. 
The last important result of Section~\ref{Families of Complete Intersections}
is Theorem~\ref{sturm} which proves the existence of an 
open non-empty semi-algebraic subscheme of the $I$-optimal subscheme 
over which the number of real zeros is constant.

Having described a good subscheme of the parameter space 
over which we are allowed to move, and hence over which we 
can perturb our data, we pass in Section~\ref{Condition Numbers} 
to the next problem and  concentrate
our investigation on a single point of the zero set. 
After some preparatory results, 
we introduce a local condition number
(see Definition~\ref{LocalCondNumb}) and with its help we 
prove Theorem~\ref{theoremCN} 
which has the merit of fully generalizing a classical result in 
numerical linear algebra (see Remark~\ref{remarkCN}).

The subsequent short 
Section~\ref{Optimization of the local condition number}
illustrates how to manipulate the equations in order to 
lower, and sometimes to minimize, the local condition number
(see Proposition~\ref{scalingCN}). Then we concentrate on the 
case of the matrix 2-norm and  show how to 
achieve the minimum when the  polynomials involved  
have equal degree (see Proposition~\ref{min2norm}).
The final Section~\ref{Experiments} describes  examples
which indicate that our approach is good, in particular we see that when 
the local condition number is lowered, indeed the corresponding 
solution is more stable. 

This paper reports on the first part of a wider investigation.
Another paper is already planned to describe how to deal with 
global condition numbers and how to generalize our method 
to the case where the polynomials  involved  have arbitrary degrees.

All the supporting computations were performed 
with \cocoa\ (see~\cite{Co}).
We thank Marie-Fran\c coise Roy and Saugata Basu for 
some help in the proof of Theorem~\ref{sturm}.



\section{Families of Complete Intersections}
\label{Families of Complete Intersections}

Given a zero-dimensional 
smooth complete intersection $\X$, we want to embed it into
a family of zero-dimensional complete intersections and study
when and how~$\X$  can move inside the family. 
In particular, we study the locus of the parameter-space 
over which the fibers are smooth with the same number 
of points as~$\X$, and we give special emphasis
to the case of real points.

\medskip

We start the section by recalling some definitions.
The notation is borrowed
from~\cite{KR1} and~\cite{KR2}, in particular
we let~$x_1, \dots, x_n$ be
indeterminates and let $\mathbb T^n$
be the monoid of the power products in the
symbols $x_1, \dots, x_n$.
Most of the times, for simplicity we use the notation
$\x = x_1, \dots, x_n$.
If~$K$ is a field, the multivariate
polynomial ring~$K[\x]=K[x_1,\dots,x_n]$ is
denoted by~$P$, and if $f_1(\x), \dots, f_k(\x)$
are polynomials in $P$,
the set~$\{f_1(\x), \dots, f_k(\x)\}$ is denoted by~$\f(\x)$
(or simply by~$\f$).
Finally, we denote the {\it polynomial system}\/
associated to~$\f(\x)$  by~$\f(\x)=0$ (or simply by $\f=0$), 
and we say that the system is zero-dimensional if 
the ideal generated by~$\f(\x)$ is zero-dimensional
(see~\cite{KR1}, Section~3.7).

Easy examples show that, unlike the homogeneous case,
in the inhomogeneous case regular sequences are not 
independent of the order of their entries. For instance,  
if $f_1 = y(x+1)$, $f_2 = z(x+1)$, $f_3 = x$, 
then $(f_1, f_2, f_3)$  is not a regular sequence, 
while $(f_3, f_1, f_2)$ is such. However,
we prefer to avoid a distinction between these cases,
and we call them {\it complete intersections}.
In other words, we use the following definition.

\begin{definition}
Let $t$ be a positive integer, let~$\f(\x)$ be a set
of~$t$ polynomials  in $P=K[x_1, \dots, x_n]$
and let $I$ be the ideal generated by $\f(\x)$.
\begin{itemize}
\item[(a)] The set~$\f(\x)$ (and the ideal~$I$)
is called a {\bf complete intersection}  if the equality 
$\dim(P/I) = n-t$ holds.

\item[(b)]  The set~$\f(\x)$  (and the ideal~$I$)
is called a {\bf zero-dimensional complete intersection}
if it is a complete intersection and $t=n$.
\end{itemize}

\end{definition}

\noindent 
Let $n$ be a positive integer,
let~$P$ denote the polynomial
ring~$K[x_1, \dots,x_n]$, 
let $\f(\x)=\{f_1(\x),\ldots,f_n(\x)\}$
be a zero-dimensional complete intersection, and let~$I$ 
be the ideal of~$P$ generated by~$\f(\x)$.
We let~$ m$ be a positive integer
and let $\aaa = (a_1, \dots, a_m)$
be an $m$-tuple of indeterminates which will play the role of
parameters. If $F_1(\aaa, \x), \ldots, F_n(\aaa, \x)$ are polynomials
in~$K[\aaa, \x]$ we let~${F(\aaa, \x)=0}$ be the corresponding 
family of systems of  equations 
parametrized by~$\aaa$, and 
the ideal generated by~$F(\aaa,\x)$ in $K[\aaa,\x]$
is denoted by $I(\aaa,\x)$. 
If the scheme of the $\aaa$-parameters is $\Sc$,
then there is a $K$-algebra homomorphism
$\phi: K[\aaa] \To K[\aaa,\x]/I(\aaa,\x)$ or, equivalently, 
a morphism of schemes
$\Phi: \Spec(K[\aaa,\x]/I(\aaa,\x)) \To \Sc$.

Although it is not strictly necessary for the theory, for our applications
it suffices to consider independent parameters. Here is 
the formal definition.

\begin{definition}\label{independparams}
If $\Sc= \mathbb A^m_K$ and 
${I(\aaa,\x)\cap K[\aaa] = (0)}$, then  the parameters~$\aaa$ are 
said to be  {\bf independent} with respect to $F(\aaa,\x)$, or simply 
independent if the context is clear.
\end{definition}

The first important step is to embed the 
system~$\f(\x)=0$ into a  family, but we must be careful 
and exclude families of the following type.

\begin{example}\label{bad}
Consider the family ${F(a,\x) = \{ x_1(ax_2+1), x_2(ax_2+1)\}}$.
It is a zero dimensional complete intersection only for $a = 0$
while the generic member is positive-dimensional.
\end{example}

\begin{definition}
Let~$\f(\x)$ be a set of polynomials  in $P=K[x_1, \dots, x_n]$ so
that~$\f(\x)$ is a zero-dimensional complete intersection
and let~$F(\aaa, \x)$ be a family parametrized by~$m$ 
independent parameters~$\aaa$. 
We say that $F(\aaa, \x)$ (and similarly $K[\aaa,\x]/I(\aaa,\x)$ 
and $\Spec(K[\aaa,\x]/I(\aaa,\x))$) is 
a {\bf generically zero-dim\-ensional 
family containing $\f(\x)$}, if~$\f(\x)$ is a member of the 
family and the generic member of the family is
a zero-dimensional complete intersection.
\end{definition}

A  theorem called {\it generic flatness}\/ (see~\cite{E}, Theorem 14.4)
prescribes the existence of a non-empty Zariski-open
subscheme $\U$ of $\Sc$ over which the morphism 
$\Phi^{-1}(\U) \To \U$ is {\it flat}. In particular, it is possible 
to explicitly compute a subscheme over which the morphism is free.
To do this, Gr\"obner bases reveal themselves as a fundamental tool.

\begin{definition}\label{iflat}
Let $F(\aaa, \x)$ be a generically zero-dimensional family 
which contains a zero-dimensional complete intersection~$\f(\x)$.
Let $\Sc=\mathbb A^m_K$ be the scheme of the independent 
$\aaa$-parameters
and let $\Phi: \Spec(K[\aaa,\x]/I(\aaa,\x))  \To \Sc$ be the associated
morphism of schemes.
A dense Zariski-open subscheme~$\U$ of~$\Sc$
such that~${\Phi^{-1}(\U) \To \U}$ is
free (flat, faithfully flat), is said to be an~{\bf $I$-free 
($I$-flat, $I{\bf-faithfully\ flat}$}) subscheme of~$\Sc$ or
simply an~$I$-free ($I$-flat, $I$-faithfully flat) scheme.
\end{definition}

\begin{proposition}\label{flatness}
With the above assumptions and notation,
let $I(\aaa, \x)$ be the ideal 
generated by~$F(\aaa,\x)$ in $K[\aaa,\x]$, let $\sigma$ 
be a term ordering on~$\mathbb T^n$, let~$G(\aaa,\x)$ be the 
reduced~$\sigma$-Gr\"obner basis of the ideal~$I(\aaa, \x)K(\aaa)[\x]$, 
let~$d(\aaa)$ be the least common multiple of all the denominators  
of the coefficients of the polynomials in $G(\aaa,\x)$, 
and let~$T =\mathbb T^n\setminus \LT_\sigma(I(\aaa,\x)K(\aaa)[\x])$.
\begin{itemize}
\item[(a)]  The open subscheme $\U$ of $\mathbb A^m_K$ 
defined by $d(\aaa)\ne 0$ is~$I$-free.

\item[(b)] The multiplicity of each fiber over $\U$ coincides with
the cardinality of~$T$.
\end{itemize}
\end{proposition}

\begin{proof}
The assumption that $F(\aaa,\x)$ is a generically zero-dimensional 
family implies that 
${\rm Spec}\big(K(\aaa)[\x]/I(\aaa,\x)K(\aaa)[\x]\big) 
\To  {\rm Spec}(K(\aaa))$ 
is finite, in other words that $K(\aaa)[\x]/I(\aaa,\x)K(\aaa)[\x]$ is a 
finite-dimensional~$K(\aaa)$-vector space.
A standard result in Gr\"obner basis theory 
(see for instance~\cite{KR1}, Theorem 1.5.7) 
shows that the residue classes of the elements 
in~$T$ form a~$K(\aaa)$-basis 
of this vector space. We denote by $\U$ the open subscheme 
of~$\mathbb A^m_K$ defined by $d(\aaa) \ne 0$. 
For every point in $\U$, the given 
reduced Gr\"obner basis evaluates to the reduced 
Gr\"obner basis of the corresponding ideal. 
Therefore the leading term ideal is the 
same for all these fibers, and so is its complement~$T$. 
If we denote by~$K[\aaa]_{d(\aaa)}$ 
the localization of~$K[\aaa]$ at the 
element  $d(\aaa)$ and by $I(\aaa,\x)^e$ the 
extension of the ideal~$I(\aaa,\x)$ to the 
ring~$K[\aaa]_{d(\aaa)}$, 
then  $K[\aaa]_{d(\aaa)}[\x]/I(\aaa,\x)^e$ turns out 
to be a free~$K[\aaa]_{d(\aaa)}$-module.
So claim (a) is proved.
Claim (b) follows immediately from (a).
\end{proof}

\begin{remark}\label{varie}
We collect here a few remarks about  this proposition.
First of all, we observe that the term ordering  $\sigma$ can be 
chosen arbitrarily.
Secondly, for every $\alpha\in \U$ let~$L_\alpha$ 
be the leading term ideal 
of the corresponding ideal $I_\alpha$.
If $\sigma$ is a  degree-compatible term ordering, 
then $L_\alpha$ is  is also 
the leading term ideal of the 
homogenization~$I_\alpha^{\rm hom}$ of $I_\alpha$ 
(see~\cite{KR2}, Proposition 5.6.3 and its proof).
\end{remark}

\goodbreak

\begin{example}\label{secondflat}
We consider the ideal ${I = (f_1, g)}$ of $K[x,y]$
where~$f_1 = x^3-y$, 
$g = x(x-1)(x+1)(x-2)(x+2)(x-3)(x+3)(x+13)(x^2+x+1)$.
We check that $I = (f_1,f_2)$
where $f_2 = xy^3 + 504x^2y - 183xy^2 + 14y^3 - 504x^2 
+ 650xy - 147y^2 - 468x + 133y$.
It is a zero-dimensional complete intersection and we 
embed it into the family~$I(\aaa,\x) = (ax^3-y, g)$. 
If we pick $\sigma = {\tt Lex}$ with $y>x$ and
perform the computation as suggested by the 
proposition, we get the freeness of 
the family for all $a$. Instead, we get the freeness of the 
family~$I(\aaa,\x) = (ax^3-y, f_2)$ for~$a\ne0$ (see a further 
discussion in Example~\ref{continued}).
\end{example}

\begin{example}\label{twopoints}
We let  $P = \C[x]$, the univariate polynomial ring,  and 
embed the ideal~$I$ generated by the following polynomial
${x^2-3x+2}$ into the generically zero-dimensional
family~${F(\aaa, x) = \{a_1x^2 - a_2x + a_3\}}$.
Such  family is given by the canonical $K$-algebra
homomorphism 
$$
\phi:\!  \C[\aaa]\! \To\! \C[\aaa, x]/(a_1, a_2, a_3)/(a_1x^2 - a_2x + a_3)
$$ 
It is a zero dimensional complete intersection~for

$\{\alfa \in \C^3 \ | \   \alpha_1\ne 0\}\  \cup \
\{\alfa \in \C^3 \ | \    \alpha_1=0, \ \alpha_2 \ne 0\}$.

\noindent It represents two distinct smooth points for

$\{\alfa \in \C^3 \ | \   \alpha_1 \ne 0, \
\alpha_2^2-4\alpha_1\alpha_3 \ne 0\}$.

\noindent It represents
a smooth point for
$\{\alfa \in \C^3 \ | \   \alpha_1=0,\ \alpha_2\ne 0\}$.

\noindent It is not a zero-dimensional complete intersection for
$\{\alfa \in \C^3 \ | \   \alpha_1=0,\ \alpha_2=0\}$.
\end{example}

This kind of examples motivates the following definition.

\begin{definition}\label{ismooth}
Let $F(\aaa, \x)$ be a generically zero-dimensional family 
containing a zero-dimensional complete intersection~$\f(\x)$.
Let $\Sc=\mathbb A^m_K$ be the scheme of the 
independent~$\aaa$-parameters
and let $\Phi\!\!: \Spec(K[\aaa,\x]/I(\aaa,\x))  \To \Sc$ be the 
associated morphism of schemes.
A dense Zariski-open subscheme~$\U$ of~$\Sc$
such that~$\Phi^{-1}(\U) \To \U$ is {\bf smooth}, i.e.\ all the 
fibers of~$\Phi^{-1}(\U) \To \U$
are zero-dimensional smooth complete intersections, 
is said to be an~{\bf $I$-smooth} subscheme of $\Sc$ or
simply an~$I$-smooth scheme.
\end{definition}

For instance in Example~\ref{twopoints}  we 
have the equality ~$\Sc = {\mathbb A}_\C^3$ 
and the following open set
$\U=\{\alfa \in \C^3 \ | \   \alpha_1 \ne 0, \
\alpha_2^2-4\alpha_1\alpha_3 \ne 0\}$
is $I$-smooth.

\begin{remark}
We observe that a dense $I$-smooth scheme may 
not exist. It suffices to consider the ideal $I = (x-1)^2$ 
embedded into the family $(x-a)^2$. 
In any event, a practical way to find one, if there is one, 
is via Jacobians, as we are going to show.
\end{remark}

\begin{theorem}\label{jacobian}
Let $F(\aaa, \x)$ be a generically zero-dimensional family 
containing a zero-dimensional complete intersection~$\f(\x)$.
We let $\Sc=\mathbb A^m_K$ be the scheme of the 
independent~$\aaa$-parameters, let $I(\aaa, \x)$ be the ideal 
generated by~$F(\aaa,\x)$ in~$K[\aaa,\x]$,
let~$D(\aaa, \x)\!=\!\det(\Jac_F(\aaa, \x))$ be the determinant of
the Jacobian matrix of~$F(\aaa, \x)$ with respect to the 
indeterminates~$\x$, let
$J(\aaa,\x)$ be the ideal sum~$I(\aaa,\x) + 
(D(\aaa,\x))$ in~$K[\aaa,\x]$, and 
let $H$ be the ideal in~$K[\aaa]$ defined by the 
equality~$H = J(\aaa,\x) \cap K[\aaa]$.
\begin{itemize}
\item[(a)] There exists an $I$-smooth subscheme 
of $\Sc$ if and only if $H\ne (0)$.
\item[(b)] If $0\ne h(\aaa) \in H$ then the open subscheme of $\Sc$ 
defined by the inequality $h(\aaa)\ne 0$ is $I$-smooth.
\end{itemize}
\end{theorem}

\begin{proof}
To prove one implication of claim (a), and simultaneously 
claim  (b), we assume that $H \ne (0)$ and  let $0\ne h(\aaa) \in H$.
We have an equality of type $h(\aaa) = 
a(\aaa,\x) f(\aaa,\x) + b(\aaa,\x) D(\aaa,\x)$
with~$f(\aaa,\x) \in I(\aaa,\x)$, and hence an equality
$1 = \frac{a(\aaa,\x)}{h(\aaa)}f(\aaa,\x) + 
\frac{b(\aaa,\x)}{h(\aaa)}D(\aaa,\x)$
in~$J(\aaa, \x)K(\aaa)[\x]$.
For every $\alpha\in\Sc$ such that~$h(\alpha) \ne 0$ 
the equality implies that 
the corresponding complete intersection has no common zeros 
with the determinant of its Jacobian matrix, hence it is smooth.
Conversely, assume that $H=(0)$. Then the canonical $K$-algebra 
homomorphism~$K[\aaa] \To K[\aaa,\x]/J(\aaa,\x)$ is injective 
and hence it induces a 
morphism $\Spec\big(K[\aaa,\x]/J(\aaa,\x)\big) \To \mathbb A^m_K$ 
of affine schemes which is dominant. It means that for a 
generic point of~$\mathbb A^m_K$, the 
scheme $\Spec\big(K[\aaa,\x]/J(\aaa,\x)\big)$ 
is not empty, and hence the corresponding complete 
intersection $\Spec\big(K[\aaa,\x]/I(\aaa,\x)\big)$ is not smooth.
\end{proof}

%

The following example illustrates these results.

\begin{example}\label{firstflat}
Let us consider the polynomials 
$f_1 = x_1^2+x_2^2-1$, ${f_2 = x_2^2+x_1}$ in~$\ \C[x_1,\!x_2]$
and the ideal ${I \!=\! (f_1, f_2)}$ generated by them.
It is a zero-dimensional complete intersection and we 
embed it into  $I(\aaa,\x) = ({x_1^2+a_1x_2^2-1},\  \ x_2^2+a_2x_1)$. 
It is a free family over $\mathbb A^2_\C$, and the 
multiplicity of each fiber is $4$. 
We compute $D(\aaa,\x)=\det(\Jac_F(\aaa, \x))$
and get $D(\aaa,\x) = -2a_1a_2x_2 + 4x_1x_2$. 
We let 
$$J (\aaa,\x)= I(\aaa,\x) +(D(\aaa,\x) )=
(x_1^2 + a_1x_2^2 -1,\ x_2^2+a_2x_1, \ -2a_1a_2x_2 + 4x_1x_2)$$
A computation with \cocoa\ of  $\ {\tt Elim}([x_1,x_2], J)$ 
yields~$(\tfrac{1}{2}a_1^2a_2^3+2a_2)$, and 
hence $J(\aaa,\x)\cap K[\aaa] = (\tfrac{1}{2}a_1^2a_2^3+2a_2)$.
According to the theorem, if $ \U$ is the complement 
in~$\mathbb A^2_\C$ of the curve  defined 
by $\tfrac{1}{2}a_1^2a_2^3 + 2a_2=0$, 
then~$\U$ is an~\hbox{$I$-smooth}  
subscheme of~$\mathbb A^2_\C$.  
On the other hand, the curve has three components, $a_2=0$,
and ${a_1a_2\pm 2i=0}$. If $a_2=0$ then  the 
corresponding ideal is $(x_1^2-1,x_2^2)$ which is not smooth. 
If $a_1a_2\pm 2i=0$, then the corresponding ideals
are 
$(x_1^2 \mp \tfrac{2i}{a_2}x_2^2 - 1,\  x_2^2+a_2x_1)$
which can be written as
$((x_1\pm i)^2,\  x_2^2+a_2x_1)$
and hence are not smooth.

Let us now consider  the zero-dimensional complete 
intersection described by the  ideal $I = (f_1, f_2)$ 
where $f_1 = x_1^2+x_2^2$, $f_2 = x_2^2+x_1$.
We embed it into the 
family~$I(\aaa,\x) = (x_1^2-a_1x_2^2,\ x_2^2+a_2x_1)$. 
As before, it is a free family
over $\mathbb A^2_\C$, and the multiplicity 
of each fiber is $4$. 
We compute $D(\aaa,\x)=\det(\Jac_F(\aaa, \x))$
and get $D(\aaa,\x) = 2a_1a_2x_2 + 4x_1x_2$.  
The computation of~$\ {\tt Elim}([x,y], J)$ 
yields~$(0)$, and hence there is no 
subscheme of $\mathbb A^2_K$ which is $I$-smooth. 
Indeed, for $a_2\ne0$ we 
have $I (\aaa,\x)= (x_1+\frac{1}{a_2}x_2^2,\  
\frac{1}{a_2^2}x_2^4-a_1x_2^2)$ which is not smooth. 
Incidentally, we observe that also for $a_2=0$ the 
corresponding zero-dimensional complete 
intersection is not smooth.

\end{example}

The following example illustrates other subtleties 
related to the theorem.

\begin{example}{(\bf Example~\ref{secondflat} 
continued)}\label{continued}\\
We consider the family $I(\aaa,\x) = (ax^3-y, f_2)$ 
for~$a\ne0$ of Example~\ref{secondflat}, 
compute $D(\aaa,\x)=\det(\Jac_F(\aaa, \x))$
and get $D(\aaa,\x) = 9ax^3y^2 + 1512ax^4- 1098ax^3y + 126ax^2y^2 
+ 1950ax^3 - 882ax^2y + y^3 + 399ax^2 + 1008xy 
- 183y^2 - 1008x + 650y - 468$. 
We let $J (\aaa,\x)= I(\aaa,\x) +(D(\aaa,\x) )$ and get 
$J(\aaa,\x)\cap K[\aaa] = (h(\aaa))$ where {\tiny $$h(\aaa)=
a^9 -  \frac{738170716516748}{7749152384519}a^8 + 
\frac{218039463835944563500746}{91409877182005574647}a^7 - 
\frac{166557011563009981474061668}{31353587873427912103921}a^6
$$
$$
-\frac{276169260891419750846552207}{31353587873427912103921}a^5 
+ \frac{986809115998719019081678896}{31353587873427912103921}a^4 
- \frac{63247607413926237871517952}{31353587873427912103921}a^3
$$
$$
-\frac{1316764479863922379654192128}{31353587873427912103921}a^2 
+ \frac{317872550804296477704192}{13058553883143653521}a 
- \frac{974975584016793600000}{266501099655992929}
$$}
Therefore, if $\ \U$ denotes the 
complement in $\mathbb A^1_K$ of the zeros of $h(a)$, the theorem 
says that it is a Zariski-open $I$-smooth subscheme. However, 
we have already seen in Example~\ref{secondflat} that
$a=0$ (the origin is  in~$\U$) is not in the free locus: we observe
that the corresponding complete intersection is smooth, but it has 
only two points. 
The other subtlety is that the B\'ezout number of the family 
is $3\times 4=12$, but if we substitute $y = ax^3$ into $f_2$ 
we get a univariate polynomial of degree~$10$.
The two {\it missing}\/ points are at infinity. No member of the 
family represents twelve points. The final remark is that if we 
move the parameter $a$ in the locus described 
by~$a\!\cdot \!h(a) \ne 0$ we always get a smooth 
complete intersection of~$10$ points. 
If $K = \C$ the ten points have complex coordinates, some of 
them are real, but there are no values of $a$ for which all 
the $10$ points are real. The reason is that 
if ${r_1 = \frac{-1+\sqrt{3}i}{2},\  r_2= \frac{-1-\sqrt{3}i}{2}}$ are the two
complex roots of $x^2+x+1 =0$, then two of the ten points are 
$(r_1, r_1^3)$, $(r_2,r_2^3)$ which are not real points
(see Theorem~\ref{sturm} and Example~\ref{realroots}).

\end{example}

Combining Theorem~\ref{jacobian} and Proposition \ref{flatness} 
we get a   method to select a Zariski-open 
subscheme of the parameter space over which all the fibers are 
smooth complete intersections of  constant multiplicity
(see~\cite{SW05} for similar results). 
Before describing the algorithm, we need 
a definition which captures this concept.

\begin{definition}\label{ioptimal}
With the above notation, 
a dense Zariski-open subscheme~$\U$ of~$\Sc$
such $\Phi^{-1}(\U) \To \U$ is smooth and free  is said to be 
 an~{\bf $I$-optimal} subscheme of~$\Sc$. 
\end{definition}

\goodbreak
\begin{corollary}\label{algo-optimal}
Let $\Sc = \mathbb A^m_K$ and
consider the following sequence of instructions.
\begin{itemize}
\item[(1)] Compute $D(\aaa,\x) = \det({\rm Jac}_F(\aaa,\x))$. 

\item[(2)] Let $J(\aaa,\x) = I(\aaa,\x) + (D(\aaa,\x))$ and 
compute $H =J(\aaa,\x) \cap K[\aaa]$.

\item[(3)] If $H = (0)$ return {\rm  ``There is no $I$-smooth 
subscheme of $\mathbb A^m_K\,$"} and stop.

\item[(4)] Choose $h(\aaa) \in H\setminus 0$ and
let  $\U_1=\mathbb A^m_K\setminus\{\alfa \in \mathbb A^m_K \, |\, h(\alfa)=0\}$.

\item[(5)] Choose a term 
ordering $\sigma$ on $\mathbb T^n$ and compute the 
reduced $\sigma$-Gr\"obner basis $G(\aaa,\x)$ of~$I(\aaa,\x)K(\aaa)[\x]$

\item[(6)]
Let $T =\mathbb T^n\setminus \LT_\sigma(I(\aaa,\x)K(\aaa)[\x])$,
compute the cardinality of~$T$ and call it~$\mu$; then
compute the least common multiple of all the denominators  of the 
coefficients of the polynomials in~$G(\aaa,\x)$, and call it $d(\aaa)$;
finally, let ${\U_2 = \mathbb A^m_K\setminus \{\alfa \in 
\mathbb A^m_K \, |\, d(\alfa) \ne 0\}}$
and let $\U = \U_1\cap \U_2$.

\item[(7)] Return $\ \U_1$, $\U_2$, $\U$, $T$,  $\mu$.
\end{itemize}

\noindent This is an algorithm which returns $\U_1$ which is $I$-smooth,
$\U_2$ which is $I$-free, $\U$ which is $I$-optimal, $T$ which provides 
a basis as $K$-vector spaces of all the fibers over $\U_2$,
and $\mu$ which is the multiplicity of all the fibers over $\U_2$.
\end{corollary}

\begin{proof}
It suffices to combine Theorem~\ref{jacobian} and 
Proposition~\ref{flatness}.
\end{proof}

\goodbreak

\begin{example}
We consider the ideal ${I = (f_1, f_2)}$ of $K[x,y]$
where~$f_1 = xy-6$, $f_2 = x^2+y^2-13$.
It is a zero-dimensional complete intersection and we 
embed it into the family~$I(\aaa,\x) = (a_1xy+a_2,\  a_3x^2+a_4y^2+a_5)$. 
We compute the reduced {\tt DegRevLex}-Gr\"obner basis 
of $I(\aaa,\x)K(\aaa)[\x]$ and 
get
$$
\{x^2+\tfrac{a_4}{a_3}y^2+\tfrac{a_5}{a_3},\ \   xy + \tfrac{a_2}{a_1},\  \ 
y^3- \tfrac{a_2a_3}{a_1a_4}x + \tfrac{a_1a_5}{a_1a_4}y\}
$$
according to the above results, a free locus is given by $a_1a_3a_4 \ne 0$.
Now we compute  $D(\aaa,\x)=\det(\Jac_F(\aaa, \x))$
and get $D(\aaa,\x) = -2a_1a_3x^2+2a_1a_4y^2$.

We let $J (\aaa,\x)= I(\aaa,\x) +(D(\aaa,\x) )$ and 
compute $J(\aaa,\x)\cap K[\aaa]$.
We get  the principal ideal generated 
by~${a_2^2}^{\mathstrut}a_3a_4 -\tfrac{1}{4}a_1^2a_5^2$.
In conclusion, an $I$-optimal subscheme
is~$\U=\mathbb A^5_K \setminus F$ where $F$ is the closed 
subscheme defined by the 
equation~$\ a_1a_3a_4({a_2^2}^{\mathstrut}a_3a_4 
-\tfrac{1}{4}a_1^2a_5^2)=0$, and $\mu = 4$.

\end{example}

\goodbreak
\begin{definition}\label{realpoints}
We say that  {\bf a point is complex} if its coordinates 
are complex numbers, and we say that {\bf a point is real} 
if its coordinates are real numbers.

\end{definition}

The following example illustrates the fact that even
if we start with a set of real points,  a zero-dimensional 
complete intersection which contains them may also 
contain complex non-real points.

\begin{example}
Let $\X$ be the set of the $10$ real points
$\{(-1,-1),\, \! (2, 8),\,\!  (-2,\!-8),\\
(3,27),\,  (-3,-27),\,  (4,64),\,
(5,125),\,  (-5,-125),\,  (6,216),\,  (-6,-216)\}$.
A zero-dim\-ensional complete intersection
containing~$\X$ is $\{f_1, f_2\}$
where $f_1 = y-x^3$ and
$f_2 = x^2y^2 - 1/4095y^4 + 1729/15x^2y - 74/15xy^2 + 1/15y^3 -
8832/5x^2 + 5852/15xy - 10754/315y^2
+ 2160x - 4632/5y + 250560/91$.
Let $I$ denote the vanishing ideal of the $10$ points
and let $J$ denote the ideal generated by $\{f_1, \, f_2\}$.
The colon ideal~$J:I$ defines the residual
intersection. Since $J$ is the intersection of a cubic and a
quartic curve, the residual intersection is a zero-dimensional
scheme of multiplicity~2. Indeed, a computation
(performed with \cocoa) shows that
$J:I $ is generated by $(x+1/78y-87/26,\, y^2-756y+658503)$.
Since $756^2 - 4*658503 = -2062476 <0$, the two
extra points on the zero-dimensional complete
intersection are complex, non real points.
\end{example}

\begin{theorem}\label{sturm}
Let $\f(\x)$ be a zero-dimensional complete intersection 
in~$\R[\x]$ and let $\f(\aaa,\x) \in \R[\aaa,\x]$ be a zero-dimensional family 
containing $\f(\x)$. Let~$I$ be the ideal in~$\R[\x]$ generated by $\f(\x)$,
assume that there exists an $I$-optimal subscheme~$\U$ 
of~$\mathbb A^m_\R$, and let~$\alfa_I\in \U$ be the point in the parameter 
space which corresponds to~$I$. If $\mu_{\R,I}$ is the number of distinct real 
points in the fiber over~$\alfa_I$ (i.e.\ zeroes of~$I$),
then there exist an open semi-algebraic subscheme $\V$ of~$\U$ such that
for every $\alfa\in \V$ the number of real points in the fiber over~$\alfa$ 
is $\mu_{\R,I}$.
\end{theorem}

\begin{proof}
We consider the ideal $ \mathcal{I} = I(\aaa,\x)\R(\aaa)[\x]$. It is 
zero-dimensional and the field $\R(\aaa)$ is infinite. 
Since a linear change of coordinates does not change the problem,
we may assume that ~$\mathcal{I}$ is in  $x_n$-normal position
(see~\cite{KR1}, Section 3.7). Moreover, we have already 
observed (see Remark~\ref{varie}) that in 
Proposition~\ref{flatness} the choice of~$\sigma$ is arbitrary. 
We choose $\sigma ={\tt Lex}$ and hence the 
reduced ${\tt Lex}$-Gr\"obner basis
of $\mathcal{I}$ has the shape prescribed by the 
Shape Lemma (see~\cite{KR1} Theorem 3.7.25).
Therefore there exists  a univariate polynomial $h_\aaa \in \R(\aaa)[x_n]$ 
whose degree is the multiplicity of both the generic fiber and 
the fiber over~$\alfa_I$, which is the number  of 
complex zeros of~$I$. 
Due to the shape of the reduced Gr\"obner basis, a point 
is real if and only if its $x_n$-coordinate is real. Therefore 
it suffices to prove the following statement: 
given a univariate square-free polynomial~$h_\aaa \in \R(\aaa)[x_n]$ 
such that~$h_{\alfa_I}$ has exactly~$\mu_{\R,I}$ real roots, 
there exists an open 
semi-algebraic subset of~$A^m_\R$ such that 
for every point $\alfa$ in it, the polynomial $h_\alfa$ has 
exactly~$\mu_{\R,I}$ real roots. This statement follows 
from~\cite{BPR}, Theorem 5.12 where it is shown that for every root
there exists an open semi-algebraic set in~$A^m_\R$ 
which isolates the root. Since complex non-real roots have to occur 
in conjugate pairs, this implies that real roots stay real.
\end{proof}

Let us see some examples.

\begin{example}
We consider the ideal $I = (xy-2y^2+2y,\  x^2-y^2-2x)$ in~$\R[x,y]$,
and we embed it into the family~$I(\aaa,\x) = (xy-ay^2+ay,\  x^2-y^2-2x)$.
We compute the reduced {\tt Lex}-Gr\"obner 
basis of~$I(\aaa,\x)\R(\aaa)[\x]$ and get
$$\{x^2 - 2x - y^2,\ xy - ay^2 + ay, \ 
y^3 - \tfrac{2a}{a-1}y^2  + \tfrac{a^2+2a}{a^2-1}y\}
$$
Applying the algorithm illustrated in Corollary~\ref{algo-optimal}
we get an $I$-smooth subscheme of~$\mathbb A^1_\R$  
for $a(a+2) \ne 0$, and an $I$-free subscheme for $(a-1)(a+1) \ne 0$.
For $a$ different from $0, -2,\  1, -1$ we have an $I$-optimal subscheme
and the multiplicity is~$4$. 

\goodbreak

Our ideal $I$ is obtained for $a = 2$, 
and hence it  lies over the optimal subscheme. It has multiplicity $4$ 
and the four zeros are real. 

The computed {\tt Lex}-Gr\"obner basis  does not have the 
shape prescribed by the Shape Lemma, so we 
perform a linear change of coordinates by setting $x = x+y,\  y=x-y$.
We compute the reduced {\tt Lex}-Gr\"obner basis and get 
$$
\{x + 4 \tfrac{a+1}{a-1}y^3 - 2 \tfrac{a+1}{a-1}y^2  - \tfrac{3a+1}{a-1}y, \ \
y^4 - y^3 -  \tfrac{1}{2}\tfrac{a}{a+1}y^2 + \tfrac{1}{2}\tfrac{a}{a+1}y\}
$$
It has the good shape, so we can use the 
polynomial
$$
h_{\aaa} = y^4 - y^3 -  \tfrac{1}{2}\tfrac{a}{a+1}y^2 + \tfrac{1}{2}\tfrac{a}{a+1}y
= y(y-1)(y^2- \tfrac{1}{2}\tfrac{a}{a+1})$$
We get the following result.
\begin{itemize}
\item For $a < -1,\  a\ne -2$ there are $4$ real points.  
\item For $-1<a<0$ there are $2$ real points.
\item For $a>0, \ a\ne 1$ there are $4$ real points.
\end{itemize}
To complete our analysis, let us see what happens 
at the {\it bad}\/ points $0, -2,\  1, -1$.

At $0$ the primary decomposition of the ideal $I_0$ is 
$(x-2,y)\cap(y^2 + 2x, xy, x^2)$, hence the
fiber consists in the simple point $(2,0)$ and a triple point at $(0,0)$.

At $-2$ we see that $(x+\tfrac{2}{3}, \ y-\tfrac{4}{3})\cap (x,y) \cap(x-2,y^2)$ 
is the primary decomposition of the ideal $I_{-2}$, and hence the 
fiber consists in the simple point $(-\tfrac{2}{3}, \tfrac{4}{3})$,  the simple 
point $(0,0)$ and a double point at $(2,0)$.

At $-1$ the primary decomposition of the ideal $I_{-1}$ 
is $(x,y) \cap(x-2,y)$, hence the fiber consists of the two simple 
real points $(0,0)$ and $(2,0)$.

At $1$  we see that $(x,y) \cap(x-2,y)\cap(x+\tfrac{1}{4},y-\tfrac{3}{4})$ 
is the primary decomposition of the ideal $I_{1}$, hence the fiber 
consists of the three simple  real 
points $(0,0)$, $(2,0)$, $(-\tfrac{1}{4}, \tfrac{3}{4})$.
\end{example}

\begin{example}\label{realroots}
We consider the ideal $I = (xy+1,\  x^2+y^2-5)$ in~$\R[x,y]$,
and we embed it into the family~$I(\aaa,x,y) = (xy+a_1x+1,\  x^2+y^2+a_2)$. 
We compute the reduced {\tt Lex}-Gr\"obner 
basis of~$I(\aaa,\x)K(\aaa)[x,y]$ and get $G(\aaa,x,y)=\{g_1,g_2\}$ where 
\begin{eqnarray*}
g_1 &=& x - y^3 - a_1y^2 - a_2y - a_1a_2,\\
g_2 &=& y^4 + 2 a_1y^3 + (a_1^2 + a_2)y^2 
+2 a_1 a_2y + (a_1^2a_2+1)
\end{eqnarray*}
which has the shape prescribed by the Shape Lemma 
(see~\cite{KR1} Theorem 3.7.25).
There is no condition for the free locus, and
$D(\aaa,x,y)=\det(\Jac_F(\aaa, x,y)) = -2x^2 + 2y^2 + 2a_1 y$.
We let $J (\aaa,x,y)= I(\aaa,x,y) +(D(\aaa,x,y) )$ and compute $J(\aaa,x,y)\cap K[\aaa]$.
We get  the principal ideal generated 
by the following polynomial ${h(\aaa)=a_1^6a_2 + 3a_1^4a_2^2 + a_1^4+ 3a_1^2a_2^3+ 20a_1^2a_2 + a_2^4
 - 8a_2^2 + 16}$.
An~$I$-optimal subscheme
is~$\ \U=\mathbb A^4_\R \setminus F$ where $F$ is the closed 
subscheme defined by the 
equation~$h(\aaa)=0$, and we observe that $\mu = 4$.

At this point we know that for $h(\aaa) \ne 0$ 
each fiber is smooth and has multiplicity~$4$, hence it consists of $4$ 
distinct complex points. What about real points? 

\begin{center}
\includegraphics[width = .8 \textwidth]{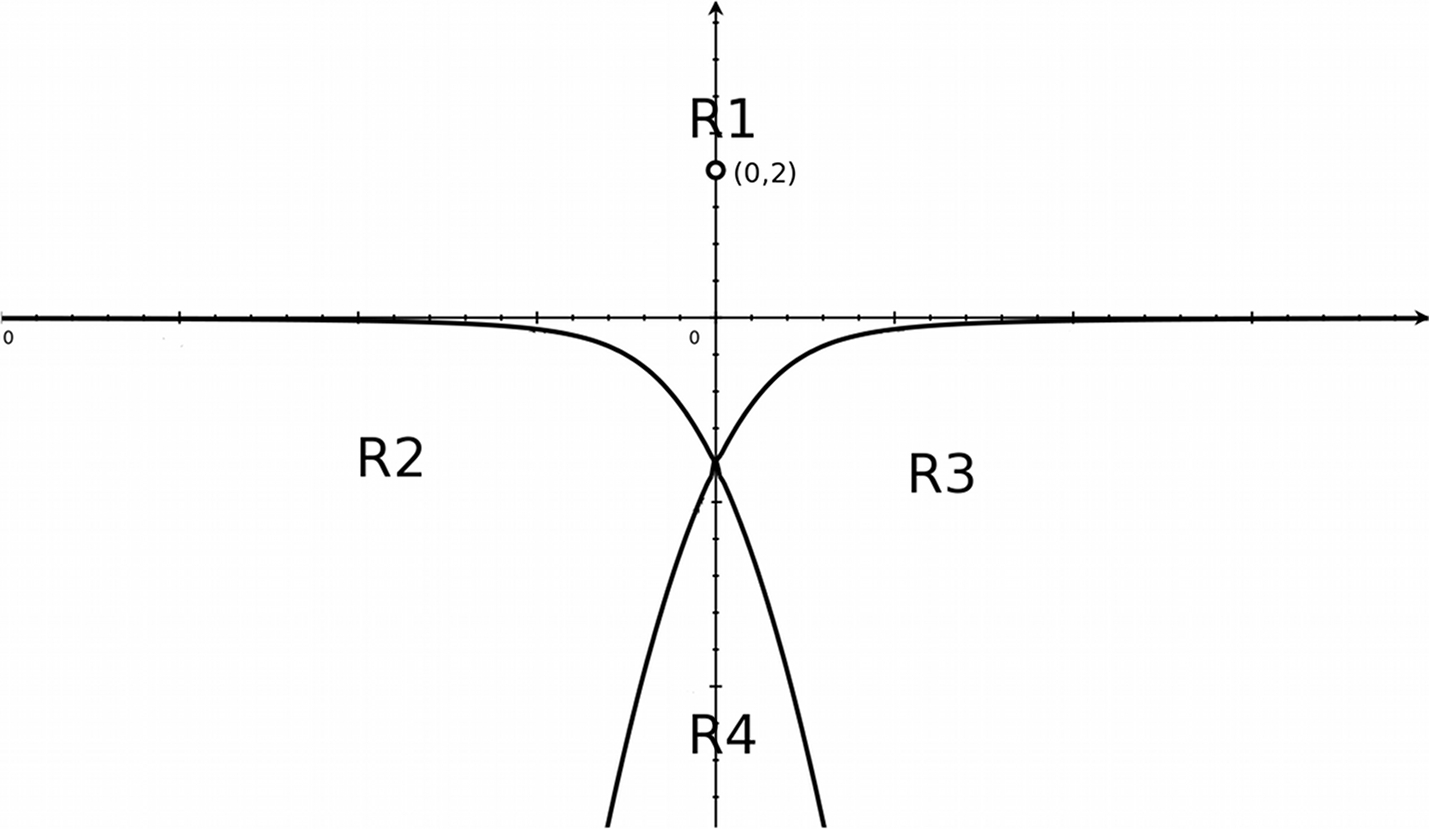}
\end{center}
The real curve defined by $h(\aaa)=0$ is shown in the above picture. 
It is the union of two branches and the isolated point $(0,2)$.
The upper region R1 (with the exception of the point $(0,2)$) 
corresponds to the ideals in the family whose zeros are four 
complex non-real points. The regions R2 and R3
correspond to the ideals whose zeros are  two complex 
non-real points and two real points.  
The region R4 corresponds to the ideals whose zeros 
are  four real points. 
To describe the four regions algebraically, we use 
the Sturm-Habicht  sequence (see~\cite{GLRR}) of~$g_2 \in \R(\aaa)[y]$. 
The leading monomials are
$y^4,\ 4y^3,\  4r(\aaa)y^2,\  -8\ell(\aaa)y,\  16h(\aaa)$ 
 where 
$r(\aaa) = a_1^2-2a_2,\  
\ell(\aaa) = a_1^4a_2+2a_1^2a_2^2+2a_1^2+a_2^3-4a_2$.
To get the total number of real roots we count the 
sign changes in the sequence at $-\infty$ and $+\infty$; 
in particular, we observe that in the parameter space  the ideal~$I$ 
corresponds to the point $(0,-5)$ which belongs to the region R4. 
We get
$$
{\rm R4} = \{\alfa \in \R^2\ | \  r(\alfa)>0, \ \ell(\alfa) <0,\  h(\alfa)>0\}
$$
which  is semi-algebraic open, not Zariski-open.
\end{example}

\goodbreak
\section{Condition Numbers}
\label{Condition Numbers}
In this section we introduce a notion of {\it condition number}\/ for
zero-dimensional smooth complete intersections in~$\R[\x]$; 
the aim is to give a measure of the sensitivity of
its real roots with respect to small perturbations of the input data,
that is  small changes of the coefficients of the involved polynomials.

The section starts with the recall of well-known facts
about numerical linear algebra.
We let $m,n$ be positive integers and let $\rm{Mat}_{m \times n}(\R)$
be the set of~$m \times n$ matrices with entries in $\R$; 
if $m=n$ we simply write $\rm{Mat}_{n}(\R)$.

\begin{definition}
Let $M=(m_{ij})$ be a matrix in $ \rm{Mat}_{m \times n}(\R)$, 
$v=(v_1, \dots, v_n)$ 
a vector in $\R^n$ and~$\| \cdot \|$ a vector norm.

\begin{itemize}
\item[(a)] Let $r \ge 1$ be a real number; the {\bf $r$-norm} 
on the vector space~$\R^n$ is defined by the formula 
$\|v\|_r = \left( \sum_{i=1}^n |v_i|^r \right)^{\frac{1}{r}} $ 
for every $v \in \R^n$.

\item[(b)]  The {\bf infinity norm} on $\R^n$ is defined by the formula
$\|v\|_\infty = {\rm max}_i |v_i|$.

\item[(c)] The {\bf spectral radius} $\rho(M)$ of the matrix $M$ 
is defined by the formula
$\rho(M) = \max_i |\lambda_i |$, where the $\lambda_i$ are 
the {\it complex}\/ eigenvalues of~$M$. 

\item[(d)] The real function defined on $\rm{Mat}_{m \times n}(\R) $ by
$M \mapsto \max_{\|v\|=1} \| Mv\|$
is a matrix norm called the {\bf matrix norm induced} by
$\| \cdot \|$. A matrix norm induced by a vector norm is
called an {\bf induced matrix norm}.

\item[(e)]  The matrix norm induced by $\|\cdot\|_1$ 
is given by the following formula
$\|M\|_1 = {\rm max}_j(\sum_i|m_{ij}|)$. 
The matrix norm induced by $\|\cdot\|_\infty$ is given 
by the formula
$\|M\|_\infty = {\rm max}_i(\sum_j |m_{ij}|)$. 
Finally, the matrix norm induced by $\|\cdot\|_2$ 
is given by the formula
$\|M\|_2 = {\rm max}_i(\sigma_i)$ where the $\sigma_i$ 
are singular values of~$M$.
\end{itemize}
\end{definition}

If no confusion arises, from now on we will use the 
symbol $\| \cdot \|$ to denote both a vector norm and a matrix norm.
We recall some facts about matrix norms 
(see for instance~\cite{BCM},~\cite{H96}).

\begin{proposition}\label{lemmaInverseNorm}
Let $M$ be a matrix in $\rm{Mat}_n(\R)$, let $I$ be the identity 
matrix of type $n$ and let~$\|\cdot \|$ be an induced matrix 
norm on $\rm{Mat}_n(\R)$.
If the matrix $I+M$ is  invertible then  $(1 - \|M\|) \: \|(I+M)^{-1}\| \leq 1$.
\end{proposition}

\begin{proposition} \label{classicalIneq2}
Let $M \in \rm {Mat}_{m \times n}(\R)$ and denote by
$M_i$ the $i$-th row of~$M$.
Let $r_1 \ge 1, r_2 \ge 1$ be real numbers such that
$\frac{1}{r_1}+\frac{1}{r_2}=1$; then
\begin{eqnarray*}
\max_i \|M_i\|_{r_2} \le \|M\|_{r_1} \le m^{1/r_1} \max_i \|M_i\|_{r_2}
\end{eqnarray*}
In particular, for $r_1=r_2=2$
\begin{eqnarray*}
\max_i \|M_i\|_2 \le \|M\|_2 \le \sqrt{m} \max_i \|M_i\|_2
\end{eqnarray*}
\end{proposition}

This introductory part ends with the recollection 
of some facts about the polynomial ring~$K[\x]$. 
In particular, given $\eta=(\eta_1,\ldots,\eta_n) \in \N^n$ 
we denote by~$|\eta|$ the number $\eta_1+\ldots+\eta_n$, 
by $\eta!$ the number~$\eta_1! \ldots \eta_n!$,  
and by $\x^\eta$ the power 
product $x_1^{\eta_1} \ldots x_n^{\eta_n}$.

\begin{definition}
Let $p$ be a point of~$K^n$; the $K$-linear map 
on~$K[\x]$ defined by $f \mapsto f(p)$ is called 
the {\bf evaluation map} associated to~$p$ and 
denoted by ${\rm ev}_{p}(f)$.
\end{definition}

\begin{definition}\label{defTaylor}
Let $d$ be a nonnegative integer, let $r \ge 1$ be a real number, 
let~$p$ be a point of~$\R^n$ and let $g(\x)$ 
be a polynomial in~$\R[\x]$.

\begin{itemize}
\item[(a)] The formal Taylor expansion of $g(\x)$ at~$p$ 
is given by the following expression:
$
g(\x) = \sum_{|\eta| \ge 0} \frac{1}{\eta!} 
\frac{\partial^\eta g}{\partial \x^\eta}(p) (\x-p)^\eta
$.
\item[(b)] The polynomial $\sum_{|\eta| \ge d} \frac{1}{\eta!} 
\frac {\partial^\eta g} {\partial \x^\eta}(p) (\x-p)^\eta$ 
is denoted by $g^{\ge d}(\x,p)$. 

\item[(c)] The {\bf  $r$-norm of $g(\x)$ at $p$} is 
defined as the $r$-norm of the 
vector $\frac{\partial g}{\partial \x}(p)$.
If $\|\frac{\partial g}{\partial \x}(p)\|_r=1$ then~$g(\x)$ 
is called {\bf unitary at~$p$}.
\end{itemize}
\end{definition}

We use the following formulation of Taylor's theorem.

\begin{proposition}\label{propTaylor}
Let $p$ be a point of~$\R^n$ and let $g(\x)$ be a 
polynomial in~$\R[\x]$. 
For every point $q \in \R^n$ we have
$$
g(q)=  g(p) + \Jac_g(p)(q-p) +  \frac{1}{2} (q-p)^t H_g(\xi) (q-p)
$$
where $\xi$ is a point of the line connecting~$p$ to~$q$ 
and $H_g(\xi)$ is the Hessian matrix of~$g$ at~$\xi$. 
\end{proposition}

Given  $\f(\x)=\{f_1(\x),\ldots,f_n(\x)\}$, a 
zero-dimensional smooth complete intersection in~$\R[\x]$, we 
introduce a notion of admissible 
perturbation of~$\f(\x)$. Roughly speaking, the polynomial 
set $\bm \varepsilon(\x)=\{\varepsilon_1(\x),\ldots,
\varepsilon_n(\x)\} \subset \R[\x]$ is considered to be an 
admissible perturbation of~$\f(\x)$ if the real solutions 
of $(\f +  \bm \varepsilon)(\x)=0$ are nonsingular and derive 
from perturbations of the real solutions of $\f(\x)=0$. 
Using the results of Section~\ref{Families of Complete Intersections} 
we formalize this concept as follows.

\begin{definition}\label{admissible}
Let $\f(\x)=\{f_1(\x),\ldots,f_n(\x)\}$ be a zero-dimensional 
smooth complete intersection 
in~$\R[\x]$, let~$\mu_{\R,I}$ be the 
number of real solutions of $\f(\x)=0$, and 
let $\bm \varepsilon(\x)=
\{\varepsilon_1(\x),\ldots,\varepsilon_n(\x)\}$ be a set of
polynomials in $\R[\x]$.
Suppose that the assumptions of Theorem~\ref{sturm} 
are satisfied, let $\mathcal V \subset \mathbb A_\R^m$ 
be an open semi-algebraic subset of~$\U$ 
such that $\alfa_I \in \V$, and for every $\alfa \in \V$ the 
number of real roots of ${\f(\alfa,\x)=0}$ is equal to~$\mu_{\R,I}$.
If there exists  $\alfa \in \V$ such  
that $(\f+\bm \varepsilon)(\x)=\f(\alfa, \x)$, 
then $\bm \varepsilon(\x)$ is called 
an {\bf admissible perturbation} of~$\f(\x)$ .
\end{definition}

Henceforth we 
let $ \bm \varepsilon(\x)=\{\varepsilon_1(\x),\ldots,\varepsilon_n(\x)\}$ 
be an admissible perturbation of~$\f(\x)$, and let
$\mathcal Z_\R(\f)=\{p_1,\ldots, p_{\mu_{\R,I}}\}$,
$\mathcal Z_\R(\f + \bm \varepsilon)=\{r_1,\ldots,r_{\mu_{\R,I}}\}$ be
the sets of real solutions of $\f(\x)=0$ 
and $(\f +  \bm \varepsilon)(\x)=0$ respectively. 
We consider each~$r_i$ as a perturbation 
of the root~$p_i$, hence we 
write $r_i=p_i + \Delta p_i$ for $i=1, \dots,\mu_{\R,I}$.

\medskip

Now we concentrate on a single element $p$ of $\mathcal Z_\R(\f)$.

\begin{corollary}\label{taylorsystem}
Let $p$ be one of the real solutions of $\f = 0$, and 
$p + \Delta p$ the 
corresponding real solution of $\f+ \bm \varepsilon = 0$. The we have
\begin{eqnarray}\label{firstOrderApprox}
0=(\f+ \bm \varepsilon)(p+ \Delta p) = 
\bm \varepsilon(p) + \Jac_{\f+ \bm \varepsilon}(p)\Delta p + 
\left( v_1(\xi_1), \ldots, v_n(\xi_n) \right)^t
\end{eqnarray}
where $\xi_1,\ldots,\xi_n$ are points on the line which
connects the points $p$ and $p + \Delta p$, and
$v_j(\xi_j) = \frac{1}{2} \Delta p^t H_{f_j+\varepsilon_j}(\xi_j)\Delta p$
for each~$j=1,\dots,n$. 
\end{corollary}

\begin{proof}
It suffices to put  $q=p+ \Delta p$, apply the formula 
of Proposition~\ref{propTaylor} to the 
polynomial system $(\f + \bm \varepsilon)(\x)$, and use the fact
that $\f(p) = 0$.
\end{proof}

\begin{example}\label{singularJacobian}
We consider the zero-dimensional smooth complete 
intersection $\f=\{f_1, f_2\}$ where $f_1=xy-6$, $f_2=x^2+y^2-13$ 
and observe that $\mathcal Z_\R(\f)=\{(-3,-2),(3,2), (-2,-3), (2,3)\}$.
The set $\f(\x)$ is embedded into the 
following family $F(\aaa,\x)=\{xy+a_1,x^2+a_2y^2+a_3\}$.
 
The semi-algebraic open set
 $$
 \V=\{\alfa \in \R^3 \:|\: \alpha_3^2-4\alpha_1^2 \alpha_2>0, 
\alpha_2>0, \alpha_3<0\}
$$
 is a subset of the $I$-optimal 
scheme  $\U = \{\alpha\in A_\R^3 | \ 
\alpha_2(\alpha_3^2-4\alpha_1^2 \alpha_2) \ne 0\}$. 
Moreover, it contains the 
point ${\alfa_I=(-6,\, 1, -13)}$, and the fiber over 
each $\alfa \in \V$ consists of~$4$ real points. 
The set $\bm \varepsilon(\x)=\{\delta_1, 
\delta_2y^2+\delta_3\}$, with $\delta_i \in \R$, is an 
admissible perturbation of~$\f(\x)$ if and only if the 
conditions $(\delta_3-13)^2-4(\delta_1-6)^2 (\delta_2+1)>0$, 
$\delta_2>-1$, and $\delta_3<13$ 
are satisfied.
Since the values $\delta_1=2$, $\delta_2=  \tfrac{5}{4}$, 
and $\delta_3=0$ satisfy the previous conditions, the polynomial 
set $\bm \varepsilon(\x)=\{2,\ \tfrac{5}{4}{y^2}^{\mathstrut}\}$ is an 
admissible perturbation of~$\f(\x)$. 
The real roots of $(\f + \bm \varepsilon)(\x)=0$ are 
\begin{eqnarray*}
\mathcal Z_\R(\f + \bm \varepsilon)=\left\{ \left(-3,-\tfrac{4}{3}\right), 
\left(3,\tfrac{4}{3}\right), (-2,-2), (2,2) \right \}
\end{eqnarray*}
For each $r_i \in \mathcal Z_\R(\f+ \bm \varepsilon)$ 
the matrix $\Jac_{\f+ \bm \varepsilon}(r_i)$ is invertible, 
as predicted by the theory.
On the contrary, by evaluating $\Jac_{\f+ \bm \varepsilon}(\x)$ 
at the third and the fourth point of $\mathcal Z_\R(\f)$ we 
obtain a singular matrix. This is an obstruction to the 
development of the theory  which suggests further restrictions
(see the following  discussion).
\end{example}

Our idea is to evaluate~$\Delta p$ using 
equation~(\ref{firstOrderApprox}) of Corollary~\ref{taylorsystem}. 
However,  while the assumption that~$\bm \varepsilon(\x)$ is an 
admissible perturbation of~$\f(\x)$ combined with the 
Jacobian criterion guarantee the non singularity of the 
matrix~$\Jac_{\f + \bm \varepsilon}(p + \Delta p)$, 
they do not imply the non singularity of the 
matrix $\Jac_{\f+ \bm \varepsilon}(p)$, as we have 
just seen in~Example~\ref{singularJacobian}. 
The next step is  to find a criterion which 
guarantees the non singularity 
of~$\Jac_{\f + \bm \varepsilon}(p)$.

\begin{lemma}\label{lemmaTau}
If $\|\Jac_\f(p)^{-1} \Jac_{ \bm \varepsilon}(p)\|< 1$ 
then $\Jac_{\f+ \bm \varepsilon}(p)$ is invertible.
\end{lemma}

\begin{proof}
By assumption $p$ is a nonsingular root of $\f(\x)=0$, 
hence $\Jac_\f(p)$ is invertible and 
so $\Jac_{\f +  \bm \varepsilon}(p) $ 
can be rewritten as
$
\Jac_{\f + \bm \varepsilon}(p) = 
\Jac_\f(p) + \Jac_{ \bm \varepsilon}(p) = \Jac_\f(p) 
\left( I + \Jac_\f(p)^{-1} \Jac_{ \bm \varepsilon}(p) \right)
$.
Consequently,  it suffices to show 
that the matrix ${I + \Jac_\f(p)^{-1} \Jac_{ \bm \varepsilon}(p)}$ 
is invertible.
And we achieve it by proving that the spectral 
radius~$\rho(\Jac_\f(p)^{-1} \Jac_{ \bm \varepsilon}(p))$ 
is smaller than 1. We have
$
\rho(\Jac_\f(p)^{-1} \Jac_{ \bm \varepsilon}(p)) 
\le \|\Jac_\f(p)^{-1} \Jac_{ \bm \varepsilon}(p)\| <1
$,
and the proof is now complete.
\end{proof}

Note that the 
requirement~$\|\Jac_\f(p)^{-1} \Jac_{ \bm \varepsilon}(p)\|< 1$ 
gives a restriction on the admissible choices 
of~$ \bm \varepsilon(\x)$, as we see in the following example.

\begin{example}{\bf (Example~\ref{singularJacobian} 
continued)}\label{exSingJacContinued}\\
Let $ \bm \varepsilon(\x)=\{\delta_1,\delta_2y^2+\delta_3\}$, 
with $\delta_i \in \R$, be an admissible perturbation of the 
zero-dimensional complete 
intersection~$\f(\x)$ of Example~\ref{singularJacobian}. 
We consider the real solution $p_4=(2,3)$ of $\f=0$ and 
compute $\|\Jac_\f(p_4)^{-1} \Jac_{\bm \varepsilon}(p_4)\|^2_2 
= \frac{117}{25}\delta_2^2$. From Lemma~\ref{lemmaTau} 
the condition $|\delta_2| < \frac{5}{39}\sqrt{13}$ is 
sufficient to have $\Jac_{\f+\bm \varepsilon}(p_4)$ invertible. 
\end{example}

From now on we assume that the hypothesis of 
Lemma~\ref{lemmaTau} is satisfied.
In order to deduce an upper bound 
for~$\|\Delta p\|$ we consider an approximation of it.

\begin{definition}
If  $\|\Jac_\f(p)^{-1} \Jac_{ \bm \varepsilon}(p)\|$ is different from $1$,
we denote the number ${1/(1-\|\Jac_\f(p)^{-1} \Jac_{ \bm \varepsilon}(p)\|)}$ 
by~$\Lambda(\f,\bm \varepsilon,p)$. Moreover, if  
equation~(\ref{firstOrderApprox}) is  truncated at the first order, 
we get the approximate 
solution~$- \Jac_{\f + \bm \varepsilon}(p)^{-1} \bm \varepsilon(p)$
which we call~$\Delta p^1$.
\end{definition}

\begin{proposition}
\label{estimateFirstOrderSolution}
Assume that ${\|\Jac_\f(p)^{-1} \Jac_{ \bm \varepsilon}(p)\| <1}$ and 
let $\|\cdot\|$ be an induced matrix norm. Then  we have
\begin{eqnarray}\label{DeltaP}
\|\Delta p^1\| \le \Lambda(\f,\bm \varepsilon,p) \;  
\|\Jac_\f(p)^{-1}\| \; \|\bm \varepsilon(p)\|
\end{eqnarray}
\end{proposition}

\begin{proof}
Lemma~\ref{lemmaTau} guarantees that the 
matrix $\Jac_{\f +\bm \varepsilon}(p)$ is invertible, so
\begin{eqnarray*}
\Delta p^1 &=&  -  \Jac_{\f+\bm \varepsilon}(p)^{-1} \bm \varepsilon(p)
= -(\Jac_\f(p)+\Jac_{\bm \varepsilon}(p))^{-1} \bm \varepsilon(p) \\
&=& - \left ( I + \Jac_\f(p)^{-1} \Jac_{\bm \varepsilon}(p) \right )^{-1} 
\Jac_\f (p)^{-1} \bm \varepsilon(p) 
\end{eqnarray*}
We apply the inequality of Proposition~\ref{lemmaInverseNorm} 
to $\Jac_\f(p)^{-1} \Jac_{\bm \varepsilon}(p)$, and get
\begin{eqnarray*}
\|\Delta p^1\|  &\le& 
\|(I + \Jac_\f(p)^{-1} \Jac_{\bm \varepsilon}(p)^{-1})\| \; 
\|\Jac_\f(p)^{-1}\| \; \|\bm \varepsilon(p)\| \\
&\le& \Lambda(\f,\bm \varepsilon,p) \; 
\|\Jac_\f(p)^{-1}\| \; \|\bm \varepsilon(p)\|
\end{eqnarray*}
which concludes the proof.
\end{proof}

We  introduce the local condition 
number of the polynomial system $\f(\x)=0$.

\begin{definition}\label{LocalCondNumb}
Let $\f(\x)$ be a zero-dimensional smooth complete intersection 
in~$\R[\x]$, let~$p$ be a nonsingular real solution of $\f(\x)=0$, 
and let $\|\cdot\|$ be a norm.
\begin{itemize}
\item[(a)]
The number $\kappa(\f,p) = \|\Jac_\f(p)^{-1}\|  \|\Jac_\f(p)\|$ 
is called the {\bf local condition number} of~$\f(\x)$ at~$p$. 

\item[(b)] If the norm is an $r$-norm, the local condition 
number is denoted
by $\kappa_r(\f,p)$.
\end{itemize}
\end{definition}

The following theorem illustrates the importance of the 
local condition number.
It depends on $f$ and $p$, not on $\varepsilon$ and 
is a key ingredient to provide an upper bound for the relative 
error $\tfrac{\|\Delta p^1\|}{\|p\|}$.

\begin{theorem}{\bf (Local Condition Number)} \label{theoremCN}\\
Let $\|\cdot\|$ be an induced matrix norm; 
under the above assumptions and the 
condition $\|\Jac_\f(p)^{-1} \Jac_{ \bm \varepsilon}(p)\|< 1$ 
we have
\begin{eqnarray}\label{UB1}
\frac{\|\Delta p^1\|}{\|p\|} \le \Lambda(\f,\bm \varepsilon,p) \; \kappa(\f, p) 
\left(  \frac{\| \Jac_{\bm \varepsilon}(p)\|}{\|\Jac_\f(p)\|} + 
 \frac{\|\bm \varepsilon(0) - \bm \varepsilon^{\ge 2}(0,p)\|}
 {\|\f(0) - \f^{\ge 2}(0,p)\|} \right)
\end{eqnarray}
\end{theorem}

\begin{proof}
By Definition~\ref{defTaylor} the evaluation of $\bm \varepsilon$ 
at $0$ 
can be expressed in this way  $\bm \varepsilon(0)=
\bm \varepsilon(p) - \Jac_{\bm \varepsilon}(p)p+ 
\bm \varepsilon^{\ge 2} (0,p)$, and 
so $\bm \varepsilon(p) =\bm \varepsilon(0)+
\Jac_{\bm \varepsilon}(p)p-  \bm \varepsilon^{\ge 2} (0,p)$.
Dividing~(\ref{DeltaP}) of 
Proposition~\ref{estimateFirstOrderSolution} by~$\|p\|$ 
we obtain
\begin{eqnarray*}
\frac{\|\Delta p^1\|}{\|p\|}  &\le& 
\Lambda(\f,\bm \varepsilon,p) \; \|\Jac_\f(p)^{-1}\|  
\: \frac{\|\bm \varepsilon(p)\|}{\|p\|}\\
&\le& \Lambda(\f,\bm \varepsilon,p) 
\; \|\Jac_\f(p)^{-1}\|  \: 
\frac{\| \Jac_{\bm \varepsilon}(p)\| \|p\| + 
\|\bm \varepsilon(0) -  \bm \varepsilon^{\ge 2} (0,p)\|}{\|p\|}\\
& = & \Lambda(\f,\bm \varepsilon,p)  
\|\Jac_\f(p)^{-1}\|  \left ( \|\Jac_{\bm \varepsilon}(p)\| + 
\frac{\|\bm \varepsilon(0) -  
\bm \varepsilon^{\ge 2} (0,p)\|}{\|p\|} \right )
\end{eqnarray*}
Using again Definition~\ref{defTaylor} we 
express $\f(0)= \f(p) - \Jac_\f(p) p +  \f^{\ge 2} (0,p)$; 
since $\f(p)=0$ we have $\|\f(0) -  \f^{\ge 2} (0,p)\| = 
\|\Jac_\f(p) p\| \le \|\Jac_\f(p)\| \|p \|$ from which 
$$
\frac{1}{\|p\|} \le \frac{\|\Jac_\f(p)\|}{\|\f(0) -  
\f^{\ge 2} (0,p)\|}
$$ 
We combine the inequalities to obtain
\begin{eqnarray*}
\frac{\| \Delta p^1\|}{\|p\|}  &\le& 
\Lambda(\f,\bm \varepsilon,p)\|\Jac_\f(p)^{-1}\| 
\left ( \|\Jac_{\bm \varepsilon}(p)\| + \|\Jac_\f(p)\| 
\frac{\|\bm \varepsilon(0) -  
\bm \varepsilon^{\ge 2} (0,p)\|}{\|\f(0) -  \f^{\ge 2} (0,p)\|} \right )\\
&\le& \Lambda(\f,\bm \varepsilon,p)  
\|\Jac_\f(p)^{-1}\|  \|\Jac_\f(p)\|  
\left ( \frac{\|\Jac_{\bm \varepsilon}(p)\|}{\|\Jac_\f(p)\|} + 
\frac{\|\bm \varepsilon(0) -  
\bm \varepsilon^{\ge 2} (0,p)\|}{\|\f(0) -  \f^{\ge 2} (0,p)\|}   \right )
\end{eqnarray*}
and the proof is concluded.
\end{proof}

\goodbreak

The following remark contains observations 
about the local condition number.

\begin{remark}\label{remarkCN}
We call attention  to the following observations.
\begin{itemize}
\item[(a)] The notion of local condition number 
given in Definition~\ref{LocalCondNumb} is a 
generalization of the classical notion of condition 
number of linear systems (see~\cite{BCM}).
In fact, if $\f(\x)$ is linear, that 
is $\f(\x)=A \x-b$ with $A \in \rm{Mat}_n(\R)$ 
invertible, and $\mathcal Z_\R(\f)=\{p\}=\{A^{-1}b\}$, 
then $\kappa(\f,p)$ is the classical condition 
number of the matrix~$A$. In fact $\Jac_\f(\x)=A$, and so
$
\kappa(\f,p) = \|\Jac_\f(p)^{-1}\| \|\Jac_\f(p)\| =  \|A^{-1}\| \|A\|
$.
Further, if we consider the 
perturbation $\bm \varepsilon(\x) = \Delta A \x - \Delta b$, 
relation (\ref{UB1}) becomes 
$$ \label{classical}
\frac{\|\Delta p\|}{\|p\|} \le \frac{1}{1- \|A^{-1}\| \; 
\|\Delta A\|}  \|A^{-1}\| \; \|A\| \left(  \frac{\|\Delta A\|}{\|A\|} +  
\frac{\|\Delta b\|}{\|b\|}\right) \eqno {(4)}
$$
which is the relation that quantifies the sensitivity 
of the $Ax=b$ problem (see~\cite{BCM}, Theorem~4.1).

\item[(b)] Using any induced matrix norm, the 
condition number $\kappa(\f,p)$ turns out to 
be greater than or equal to $1$. 
In particular, using the $2$-norm we have 
$\kappa_2(\f,p)= 
\frac{\sigma_{\max}(\Jac_\f(p))}{\sigma_{\min}(\Jac_\f(p))}$; 
in this case the local condition number 
attains its minimum, that is $\kappa_2(\f,p)=1$, 
when $\Jac_\f(p)$ is orthonormal.

\item[(c)] The condition number $\kappa(\f,p)$ is 
invariant under a scalar multiplication of the
 polynomial system $\f(\x)$ by a unique 
 nonzero real number $\gamma$.
On the contrary, $\kappa(\f,p)$ is not invariant 
under a generic scalar multiplication of each 
polynomial~$f_j(\x)$ of~$\f(\x)$. 
The reason is that if we multiply each~$f_j(\x)$ 
by a nonzero real number $\gamma_j$ we 
obtain the new polynomial 
set ${\g(\x) = \{\gamma_1 f_1(\x),\ldots,\gamma_n f_n(\x)\}}$ 
whose condition number at $p$ is
\begin{eqnarray*}
\kappa(\g, p) = \|\Jac_\f(p)^{-1} \Gamma^{-1}\| \:
 \| \Gamma \Jac_\f(p) \| \neq \kappa(\f,p)
\end{eqnarray*}
where $\Gamma={\rm diag}(\gamma_1,
\ldots,\gamma_n) \in \rm{Mat}_n(\R)$ 
is the diagonal matrix with
entries~$\gamma_1,\ldots,\gamma_n$.

\item[(d)]
It is interesting to observe that if $p$ is  the origin 
then Formula~(\ref{UB1}) of the theorem is not applicable. 
However, one can translate $p$ away from the origin, 
and the nice thing is that  the local condition number 
does not change.
\bigskip
\end{itemize}
\end{remark}

\section{Optimization of the local condition number}
\label{Optimization of the local condition number}
In this section we introduce a strategy to improve the 
numerical stability of zero-dimensional smooth 
complete intersections. 
Let $\f(\x)=\{f_1(\x),\ldots,f_n(\x)\}$ be a 
zero-dimensional smooth complete intersection 
in $\R[\x]$, and let~$I$ be the ideal of~$\R[\x]$ 
generated by~$\f(\x)$; our aim is to find an alternative 
representation of~$I$ with minimal local  condition number. 

Motivated by Remark~\ref{remarkCN}, item (b) and (c), we consider 
the strategy of resizing each polynomial of~$\f(\x)$,
and study its effects on the condition number.
The following proposition shows that rescaling 
each~$f_j(\x)$ so that $\frac{\partial f_j}{\partial \x}(p)$ 
has unitary norm  is a nearly optimal, in some cases optimal, strategy. 
The result is obtained by adapting the method of 
Van der Sluis (see~\cite{H96}, Section 7.3) to the polynomial case.

\begin{proposition}\label{scalingCN}
Let~$p$ be a nonsingular real solution of  $\f(\x)=0$,
let $r_1 \ge 1, r_2 \ge 1$ 
be real numbers such that ${\frac{1}{r_1}+ \frac{1}{r_2}=1}$, 
including the pairs $(1, \infty)$ and $(\infty, 1)$, 
let $\gamma =(\gamma_1,\ldots,\gamma_n)$ 
be an $n$-tuple of nonzero real numbers, and
let $\g_\gamma(\x)$, ${\bm u}(\x)$ be the polynomial systems 
defined by $\g_\gamma(\x)=\{\gamma_1 f_1(\x),\ldots,\gamma_n f_n(\x)\}$ 
and~${\bm u}(\x) = \{ \|\frac{\partial f_1}{\partial \x}(p)\|_{r_2}^{-1} f_1(\x),
\ldots,\|\frac{\partial f_n}{\partial \x}(p)\|_{r_2}^{-1} f_n(\x)\}$. 

\begin{itemize}
\item[(a)] We have the inequality
$\kappa_{r_1} ({\bm u},p) \le n^{1/r_1}  \kappa_{r_1}(\g_\gamma,p)$.

\item[(b)] In particular, if  $(r_1, r_2) = (\infty, 1)$ we have the equality
$$\kappa_\infty({\bm u},p) = {\rm min}_\gamma  \kappa_\infty(\g_\gamma,p)$$
where
${\bm u}(\x) = \{ \|\frac{\partial f_1}{\partial \x}(p)\|_1^{-1} f_1(\x),
\ldots,\|\frac{\partial f_n}{\partial \x}(p)\|_1^{-1} f_n(\x)\}$.
\end{itemize}
\end{proposition}

\begin{proof}
Let $\Gamma=\rm{diag(\gamma_1,\ldots,\gamma_n)}$ 
and $D=\rm{diag}(\|\frac{\partial f_1}{\partial \x}(p)\|_{r_2}^{-1},
\ldots,\|\frac{\partial f_n}{\partial \x}(p)\|_{r_2}^{-1})$; 
then $\Jac_{\g_\gamma}(\x) = \Gamma \Jac_\f(\x)$ 
and $\Jac_{\bm u}(\x)=D \Jac_\f(\x)$. 
The condition numbers of $\g_\gamma(\x)$ and ${\bm u}(\x)$ at~$p$ 
are given by
\begin{eqnarray*}
\kappa_{r_1} (\g_\gamma,p) &
=& \| (\Gamma \Jac_\f(p))^{-1}\|_{r_1} \|\Gamma \Jac_\f(p)\|_{r_1}\\ 
\kappa_{r_1} ({\bm u},p) &
=& \| (D \Jac_\f(p))^{-1} \|_{r_1} \|D \Jac_\f(p)\|_{r_1}
\end{eqnarray*}
 From Proposition~\ref{classicalIneq2} we have
\begin{eqnarray*}
\|D \Jac_\f(p)\|_{r_1} &\le& n^{1/r_1} 
\max_i \|(D \Jac_\f(p))_i\|_{r_2} = n^{1/r_1}\\
\|(D \Jac_\f(p))^{-1}\|_{r_1} &=& 
\|\Jac_\f^{-1}(p) D^{-1}\|_{r_1} = \|\Jac_\f^{-1}(p) 
\Gamma^{-1} \Gamma D^{-1}\|_{r_1} \\
&\le& \|\Jac_\f^{-1}(p) \Gamma^{-1}\|_{r_1}  
\max_i \left ( | \gamma_i | \: \left\| \frac{\partial f_i}
{\partial \x}(p) \right\|_{r_2} \right)  \\
&\le& \|\Jac_\f^{-1}(p) \Gamma^{-1}\|_{r_1}
 \|\Gamma \Jac_\f(p)\|_{r_1} = \kappa_{r_1}(\g_\gamma,p)
\end{eqnarray*}
therefore $
\kappa_{r_1} ({\bm u},p) \le n^{1/r_1} \kappa_{r_1}(\g_\gamma,p)$
and $(a)$ is proved. To prove $(b)$ it suffices to use $(a)$ 
and observe that $n^{1/\infty} = 1$
\end{proof}

\begin{remark}
The above proposition implies that  the strategy of 
rescaling each polynomial~$f_j(\x)$ to make it unitary at~$p$ 
(see Definition~\ref{defTaylor}) is beneficial
for lowering the local condition number of~$\f(\x)$ at~$p$. 
This number is minimum when ${r=\infty}$, it is within 
factor $\sqrt{n}$ of the minimum when $r=2$.  
However, for $r=2$ we can do better, at least when all the 
polynomials $f_1(\x), \dots, f_n(\x)$ have equal degree.  
The idea is to use Remark~\ref{remarkCN}, item~(b) 
which says that when using the matrix $2$-norm, 
the local condition number attains its minimum when the 
Jacobian matrix is orthonormal. 
\end{remark}

\begin{proposition}\label{min2norm}
Let $\f = (f_1, \dots, f_n)$ be a smooth zero-dimensional complete 
intersection in~$\R[\x]$ such that  $\deg(f_1) = \cdots = \deg(f_n)$ 
and let $p\in \mathcal Z_\R(\f)$. 
Moreover, let $C = (c_{ij}) \in {\rm Mat}_n(\R)$ 
be an invertible matrix, and let $\g$ be defined 
by $\g^{\rm tr} = C\cdot \f^{\rm tr}$.
Then the following conditions are equivalent
\begin{itemize}
\item[(a)]  $\kappa_2(\g, p)=1$, the minimum possible.
\item[(b)]  $C^t C = (\Jac_\f(p) \Jac_\f(p)^t)^{-1}$.
\end{itemize}
\end{proposition}

\begin{proof}
We know that $\kappa_2(\g,p) =1$ if and only 
if the matrix $\Jac_\g(p)$ is orthonormal.
This condition can be expressed by the 
equality $\Jac_\g(p) \Jac_\g(p)^t = I_n$, 
that is $C \Jac_\f(p) \Jac_\f(p)^t C^t=I_n$ 
and the conclusion follows.
\end{proof}

We observe that condition $(b)$ of the proposition 
requires that  the entries of $C$
satisfy an underdetermined  system of $(n^2+n)/2$ 
independent quadratic 
equations in $n^2$ unknowns.

\section{Experiments}\label{Experiments}

In numerical linear algebra it is well-known 
(see for instance~\cite{BCM}, Ch. 4, Section~1)  that the 
upper bound given by the classical formula (4) of 
Remark~\ref{classical} (a)
is not necessarily sharp. 
Since our upper bound~(\ref{UB1})
generalizes the classical one, as shown in 
Remark~\ref{remarkCN},
we provide some experimental evidence 
that lowering the condition number not only 
sharpens the upper bound, but indeed
stabilizes the solution point.

\begin{example}
We consider the ideal $I=(f_1, f_2)$ in $\R[x,y]$ where 
\begin{eqnarray*}
f_1 &=& \tfrac{1}{4}x^2y + xy^2 + \tfrac{1}{4}y^3 + \tfrac{1}{5}x^2 - 
\tfrac{5}{8}xy + \tfrac{13}{40}y^2 + \tfrac{9}{40}x - \tfrac{3}{5}y + \tfrac{1}{40}\\
f_2 &=& x^3 + \tfrac{14}{13}xy^2 + \tfrac{57}{52}x^2 - \tfrac{25}{52}xy 
+ \tfrac{8}{13}y^2 - \tfrac{11}{52}x - \tfrac{4}{13}y - \tfrac{4}{13}
\end{eqnarray*}
It is a zero-dimensional smooth complete 
intersection with~$7$ real roots and
we consider the point $p=(0,1) \in \mathcal Z_\R(\f)$. 
The polynomial system $\f=\{f_1, f_2\}$ is unitary at~$p$ 
and its condition number is $\kappa_2(\f,p)=8$.
Using Proposition~\ref{min2norm} we construct a 
new polynomial system~$\g$ with minimal local 
condition number at~$p$. 
The new pair of  generators $\g$ is defined 
(see Proposition~\ref{min2norm}) by the following 
the formula $\g^{\rm tr}=C \cdot \f^{\rm tr}$, 
where $C=(c_{ij}) \in \Mat_2(\R)$ is an invertible 
matrix whose entries satisfy the following system
\begin{eqnarray*}
\left \{ \begin{array}{lll}
c_{11}^2 + c_{21}^2 &=& \tfrac{25}{16}\\
c_{11}c_{12} + c_{21}c_{22} &=& -\tfrac{15}{16}\\
c_{12}^2 + c_{22}^2 &=& \tfrac{25}{16}
\end{array} \right .
\end{eqnarray*}
A solution is given by $c_{11}=1$, $c_{12}
=0$, $c_{21}=\frac{63}{16}$, $c_{22}=-\frac{65}{16}$, and
we observe that the associated unitary polynomial 
system $\g=\{f_1, \frac{63}{16} f_1 - \frac{65}{16}f_2\}$ 
provides an alternative representation of~$I$ with 
minimal local condition 
number $\kappa_2(\g,p)=1$ at the point~$p$. 

Now we embed the system~$\f(x,y)$ into the 
family $F(a,x,y)=\{F_1, F_2\}$ where
\begin{eqnarray*}
F_1(a,x,y) &=& \tfrac{1}{4}x^2y + xy^2 
+ \tfrac{1}{4}y^3 + \tfrac{1}{5}x^2 - 
\tfrac{5}{8}xy + \left(\tfrac{13}{40}-a \right)y^2\\
&& + \left(\tfrac{9}{40}+a \right)x 
+ \left(- \tfrac{3}{5}+a \right)y + \tfrac{1}{40}-2a\\
F_2(a,x,y) &=& x^3 + \tfrac{14}{13}xy^2 
+ \tfrac{57}{52}x^2 - \tfrac{25}{52}xy 
+ \left(\tfrac{8}{13}+a \right)y^2\\
&& + \left(- \tfrac{11}{52}+a \right)x 
- \left(\tfrac{4}{13}+a \right)y - \tfrac{4}{13}+a^2
\end{eqnarray*}
We denote by $I_F(a,x,y)$ the ideal 
generated by~$F(a,x,y)$ in~$\R[a,x,y]$, 
compute the reduced {\tt Lex}-Gr\"obner 
basis of $I_F(a,x,y)\R(a)[x,y]$, and get 
$$
\{x  + \tfrac{l_1(a,y)}{d_F(a)}, \ \ y^9 + l_2(a,y)\}
$$
where $l_1(a,y),l_2(a,y) \in \R[a,y]$ have 
degree~$8$ in~$y$ and
$d_F(a) \in \R[a]$ has degree~$12$.
This basis has the shape prescribed by 
the Shape Lemma and a flat locus is given 
by $\{\alpha\in \R \;|\; d_F(\alpha) \neq 0\}$. 
We let $D_F(a,x,y)=\det(\Jac_F(a,x,y))$, 
$J_F(a,x,y)=I_F(a,x,y)+(D_F(a,x,y))$, 
compute $J_F(a,x,y) \cap \R[a]$, and 
we get the principal ideal generated 
by a univariate polynomial $h_F(a)$ of degree~$28$. 
An $I$-optimal subscheme 
is $\U_F=\{\alpha \in \R \;|\; d_F(\alpha)h_F(\alpha) \neq 0 \}$.
An open semi-algebraic subset~$\V_F$ of~$\U_F$ 
which contains the point  $\alpha_I=0$ and such that the 
fiber over each $\alpha \in \V_F$
consists of~$7$ real points, is given by the open 
interval $(\alpha_1,\alpha_2)$, 
where $\alpha_1<0$ and $\alpha_2>0$ are the real 
roots of $d_F(a)h_F(a)=0$ closest to the origin. 
Their approximate values are $\alpha_1=-0.00006$ 
and $\alpha_2=0.01136$.

To produce similar perturbations, we embed the 
system~$\g(x,y)$ into the family $G(a,x,y)=\{G_1, G_2\}$ where
\begin{eqnarray*}
G_1(a,x,y) &=& \tfrac{1}{4}x^2y + xy^2 + \tfrac{1}{4}y^3 
+ \tfrac{1}{5}x^2 - \tfrac{5}{8}xy + \left(\tfrac{13}{40}-a \right)y^2\\
&& + \left(\tfrac{9}{40}+a \right)x + \left(- \tfrac{3}{5}
+a \right)y + \tfrac{1}{40}-2a\\
G_2(a,x,y) &=& -\tfrac{65}{16}x^3 
+ \tfrac{63}{64}x^2y - \tfrac{7}{16}xy^2 
+ \tfrac{63}{64}y^3 - \tfrac{1173}{320}x^2 
- \tfrac{65}{128}xy\\
&& + \left(- \tfrac{781}{640}+a \right)y^2 
+ \left( \tfrac{1117}{640}+a \right)x 
+ \left(- \tfrac{89}{80}-a \right)y + \tfrac{863}{640}+a^2
\end{eqnarray*}
We denote by $I_G(a,x,y)$ the ideal generated 
by~$G(a,x,y)$ in~$\R[a,x,y]$, 
compute the reduced {\tt Lex}-Gr\"obner 
basis of $I_G(a,x,y)\R(a)[x,y]$, and get 
$$
\{x  + \tfrac{l_3(a,y)}{d_G(a)}, \; y^9 + l_4(a,y)\}
$$
where $l_3(a,y),l_4(a,y) \in \R[a,y]$ have 
degree~$8$ in~$y$ and
$d_G(a) \in \R[a]$ has degree~$12$, therefore  the basis has 
the shape prescribed by the Shape Lemma. 
A flat locus is given by $\{\alpha\in \R \;|\; d_G(\alpha) \neq 0\}$. 
We let $D_G(a,x,y)=\det(\Jac_G(a,x,y))$, $J_G(a,x,y)
=I_G(a,x,y)+(D_G(a,x,y))$ 
and compute $J_G(a,x,y) \cap \R[a]$. 
We get the principal ideal generated 
by a univariate polynomial $h_G(a)$ of degree~$28$. 
An $I$-optimal subscheme 
is $\U_G=\{\alpha \in \R \;|\; d_G(\alpha)h_G(\alpha) \neq 0 \}$.
An open semi-algebraic subset~$\V_G$ of~$\U_G$ 
containing the point  $\alpha_I=0$ and such that the 
fiber over each $\alpha \in \V_G$
consists of~$7$ real points is given by the 
open interval $(\alpha_3,\alpha_4)$, 
where $\alpha_3<0$ and $\alpha_4>0$ are 
the real roots of $d_G(a)h_G(a)=0$
closest to the origin. 
Their approximate values are  $\alpha_3=-0.00009$ 
and $\alpha_4=0.00914$.

Let $\alpha \in (\alpha_1,\alpha_4)$. 
According to Definition~\ref{admissible} the polynomial set 
$\bm \varepsilon(x,y)=\{-\alpha y^2 + \alpha x 
+ \alpha y - 2\alpha,\; \alpha y^2 + \alpha x 
- \alpha y + \alpha^2\}$ is an admissible perturbation 
of~$\f(x,y)$ and~$\g(x,y)$.
Further, since $\|\Jac_\f(p)^{-1} \Jac_{\bm \varepsilon}(p)\|_2 
= \sqrt{65} |\alpha| < 1$ and 
$\|\Jac_\g(p)^{-1} \Jac_{\bm \varepsilon}(p)\|_2 
= \sqrt{2} |\alpha| < 1$
Theorem~\ref{theoremCN} can be applied. 

We let  $q \in \mathcal Z_\R(\f + \bm \varepsilon)$ 
and $r \in \mathcal Z_\R(\g +\bm \varepsilon)$ 
be the two perturbations of the point~$p$. 
In order to compare the numerical behaviour of~$\f$ 
and~$\g$ at the real root~$p$ 
we compare the relative errors $\frac{\|q - p\|_2}{\|p\|_2}$ 
and $\frac{\|r - p\|_2}{\|p\|_2}$ 
for different values of~$\alpha$.
The first column of the following table contains the 
values of the local condition 
numbers of~$\f$ and~$\g$ at~$p$.
The second column contains the mean values of the 
upper bounds ${\rm UB}(\f,p)$ and ${\rm UB}(\g,p)$ 
given by Theorem~\ref{theoremCN}, 
computed for  $100$ random values 
of ~$\alpha \in (\alpha_1,\alpha_4)$. 
The third column contains the mean values of 
$\frac{\|q-p\|_2}{\|p\|_2}$ and $\frac{\|r-p\|_2}{\|p\|_2}$ 
for the same  values of $\alpha$.

\begin{table}[htb]
\centering
\begin{tabular}{|c|c|c|}
\hline
$\kappa_2(\f,p)$ & ${\rm UB}(\f,p)$ & $\frac{\|q-p\|_2}{\|p\|_2}$\\
\hline
$8$ & $0.1729$ & $0.000097$ \\
\hline \hline
$\kappa_2(\g,p)$ &${\rm UB}(\g,p)$ & $\frac{\|r-p\|_2}{\|p\|_2}$ \\
\hline
$1$ & $0.0275$ & $0.000023$\\
\hline
\end{tabular}
\end{table}
\noindent The fact that the mean values of $\frac{\|q - p\|_2}{\|p\|_2}$ are smaller than the 
mean values of~$\frac{\|r - p\|_2}{\|p\|_2}$ suggests that~$p$ is more stable when 
it is considered as a root of~$\g$ instead of as a root of~$\f$.
\end{example}

\begin{example}
We consider the ideal $I=(f_1, f_2, f_3)$ in $\R[x,y,z]$ 
where 
\begin{eqnarray*}
f_1 &=& \tfrac{6}{17}x^2 + xy - \tfrac{24}{85}x 
- \tfrac{8}{85}y - \tfrac{6}{85}\\ 
f_2 &=& \tfrac{39}{89}x^2 + \tfrac{70}{89}xy 
+ yz - \tfrac{39}{89}x + \tfrac{10}{89}y\\
f_3 &=& y^2 + 2xz + z^2 - z
\end{eqnarray*}
It  is a zero-dimensional smooth complete 
intersection with~$6$ real roots and
we consider the point $p=(1,0,0) \in \mathcal Z_\R(\f)$. 
The polynomial system $\f=\{f_1, f_2, f_3\}$ is unitary 
at~$p$ and its condition number is $\kappa_2(\f,p)=123$.
Using Proposition~\ref{min2norm} we construct a new 
polynomial system~$\g$ with minimal local condition 
number at~$p$. The new set $\g$ is defined 
by $\g^{\rm tr}=C \cdot \f^{\rm tr}$, 
where $C=(c_{ij}) \in \Mat_3(\R)$ is an invertible 
matrix whose entries satisfy the following system
\begin{eqnarray*}
\left \{ \begin{array}{llc}
c_{11}^2 + c_{21}^2 + c_{31}^2 &=& \tfrac{57229225}{15129}\\ 
c_{11}c_{12} + c_{21}c_{22} + c_{31}c_{32} &
=& -\tfrac{57221660}{15129}\\ 
c_{11}c_{13} + c_{21}c_{23} + c_{31}c_{33} &=& 0\\ 
c_{12}^2 + c_{22}^2 + c_{32}^2 &=& \tfrac{57229225}{15129}\\ 
c_{12}c_{13} + c_{22}c_{23} + c_{32}c_{33} &=& 0\\
c_{13}^2 + c_{23}^2 + c_{33}^2 &=& 1\\
\end{array} \right .
\end{eqnarray*}
A solution is given by $c_{11}=c_{33}=1$, $c_{12}=c_{13}
=c_{23}=c_{32}=0$, ${c_{21}=\frac{7564}{123}}$, 
${c_{22}=-\frac{7565}{123}}$. 
Therefore  the associated unitary 
polynomial system is the following
$\g=\{f_1, \frac{7564}{123}f_1 -\frac{7565}{123}f_2,
f_3\}$. It provides an alternative representation of~$I$ 
with minimal local condition number $\kappa_2(\g,p)=1$ at the point $p$. 

We embed the system~$\f(x,y,z)$ into the 
family $F(a,x,y,z)=\{F_1, F_2, F_3\}$ where
\begin{eqnarray*}
F_1(a,x,y,z) &=& \tfrac{6}{17}x^2 + (1-a^2)xy 
+ (-\tfrac{24}{85}+a)x +(-\tfrac{8}{85}-a)y + (-\tfrac{6}{85}+a^2)\\
F_2(a,x,y,z) &=& \tfrac{39}{89}x^2 + (\tfrac{70}{89}+a)xy 
+ yz + (\tfrac{39}{89}+a)x + (\tfrac{10}{89}+a)y\\
F_3(a,x,y,z) &=& y^2 + 2xz + (1-2a)z^2 + (-1+a)z
\end{eqnarray*}
We denote by $I_F(a,x,y,z)$ the ideal generated 
by~$F(a,x,y,z)$ in~$\R[a,x,y,z]$, 
compute the reduced {\tt Lex}-Gr\"obner basis 
of $I_F(a,x,y,z)\R(a)[x,y,z]$, and get 
$$
\{x  + \tfrac{l_1(a,z)}{d_F(a)}, \; y + 
\tfrac{l_2(a,z)}{d_F(a)}, \; z^9 + \tfrac{l_3(a,z)}{e_F(a)}\}
$$
where $l_1(a,z),l_2(a,z),l_3(a,z) \in \R[a,z]$ have degrees
$\deg_z(l_1)=\deg_z(l_2)=7$ and $\deg_z(l_3)=8$ while
$d_F(a) \in \R[a]$ has degree~$54$, 
and $e_F(a) \in \R[a]$ has degree~$11$.
The basis has the shape prescribed by the Shape Lemma. 
A flat locus is given 
by $\{\alpha\in \R \;|\; d_F(\alpha) e_F(\alpha) \neq 0\}$. 
We let $D_F(a,x,y,z)=\det(\Jac_F(a,x,y,z))$, 
$J_F(a,x,y,z)=I_F(a,x,y,z)+(D_F(a,x,y,z))$ 
and compute $J_F(a,x,y,z) \cap \R[a]$. 
We get the principal ideal generated 
by a univariate polynomial $h_F(a)$ of degree~$59$. 
An $I$-optimal subscheme is $\U_F=\{\alpha \in \R \;|\; 
d_F(\alpha)e_F(\alpha)h_F(\alpha) \neq 0 \}$.
An open semi-algebraic subset~$\V_F$ of~$\U_F$ 
containing the point $\alpha_I=0$ and such that 
the fiber over each $\alpha \in \V_F$
consists of~$6$ real points is given by the open 
interval $(\alpha_1,\alpha_2)$, 
where $\alpha_1<0$ and $\alpha_2>0$ are the 
real roots of $d_F(a)e_F(a) h_F(a)=0$
closest to the origin. Their approximate values 
are $\alpha_1=-0.17082$ and $\alpha_2=0.20711$.

To produce similar perturbations, 
we embed the system~$\g(x,y,z)$ into the family 
$G(a,x,y,z)=\{G_1, G_2, G_3\}$ where
\begin{eqnarray*}
G_1(a,x,y) &=& \tfrac{6}{17}x^2 + (1-a^2)xy 
+ (-\tfrac{24}{85}+a)x +(-\tfrac{8}{85}-a)y + (-\tfrac{6}{85}+a^2)\\
G_2(a,x,y) &=& -\tfrac{3657}{697}x^2 
+ (\tfrac{538}{41}+a)xy - \tfrac{7565}{123}yz 
+ (\tfrac{33413}{3485}+a)x \\
&& +(- \tfrac{44254}{3485}+a)y - \tfrac{15128}{3485}\\
G_3(a,x,y) &=& y^2 + 2xz + (1-2a)z^2 + (-1+a)z
\end{eqnarray*}
We denote by $I_G(a,x,y,z)$ the ideal 
generated by~$G(a,x,y,z)$ in~$\R[a,x,y,z]$, 
compute the reduced {\tt Lex}-Gr\"obner 
basis of $I_G(a,x,y,z)\R(a)[x,y,z]$, and get 
$$
\{x  + \tfrac{l_4(a,z)}{d_G(a)}, \; y 
+ \tfrac{l_5(a,z)}{d_G(a)}, \; z^9 + \tfrac{l_6(a,z)}{e_G(a)}\}
$$
where $l_4(a,z),l_5(a,z),l_6(a,z) \in \R[a,z]$ have degrees
$\deg_z(l_4)=\deg_z(l_5)=7$ and $\deg_z(l_6)=8$ while
$d_G(a) \in \R[a]$ has degree~$54$, 
and $e_G(a) \in \R[a]$ has degree~$11$.
The basis has the shape prescribed by the Shape Lemma.  
A flat locus is given by $\{\alpha\in \R \;|\; d_{G1}
(\alpha) d_{G2}(\alpha) \neq 0\}$. 
We let $D_G(a,x,y,z)=\det(\Jac_G(a,x,y,z))$, 
$J_G(a,x,y,z)=I_G(a,x,y,z)+(D_G(a,x,y,z))$ 
and compute $J_G(a,x,y,z) \cap \R[a]$. 
We get the principal ideal generated by a 
univariate polynomial $h_G(a)$ of degree~$59$. 
An $I$-optimal subscheme is 
$\U_G=\{\alpha \in \R \;|\; d_G(\alpha)e_G(\alpha)h_G(\alpha) \neq 0 \}$.
An open semi-algebraic subset~$\V_G$ of~$\U_G$ 
containing the point $\alpha_I=0$ and such that the 
fiber over each $\alpha \in \V_G$
consists of~$6$ real points is given by the open 
interval $(\alpha_3,\alpha_4)$, 
where $\alpha_3<0$ and $\alpha_4>0$ are the 
real roots of $d_G(a) e_G(a) h_G(a)=0$
closest to the origin. Their approximate 
values are $\alpha_3=-0.02942$ 
and $\alpha_4=0.03312$.

Let $\alpha \in (\alpha_3,\alpha_4)$. According to 
Definition~\ref{admissible} the polynomial set 
$\bm \varepsilon(x,y)=\{-\alpha^2xy 
+ \alpha x - \alpha y + \alpha^2, \; \alpha xy + \alpha x 
+ \alpha y, \; -2 \alpha z^2+ \alpha z\}$ is an 
admissible perturbation of~$\f(x,y,z)$ and~$\g(x,y,z)$.

We let  $q \in \mathcal Z_\R(\f + \bm \varepsilon)$ 
and $r \in \mathcal Z_\R(\g +\bm \varepsilon)$ 
be the two perturbations of the point~$p$. 
In order to compare the numerical behaviour of~$\f$ 
and~$\g$ at the real root~$p$ 
we compare the relative errors $\frac{\|q - p\|_2}{\|p\|_2}$ 
and $\frac{\|r - p\|_2}{\|p\|_2}$ for different values of~$\alpha$.
The first column of the following table contains the 
values of the local condition numbers of~$\f$ and~$\g$ at~$p$. 
The second column contains the mean values of 
$\frac{\|q-p\|_2}{\|p\|_2}$ and $\frac{\|r-p\|_2}{\|p\|_2}$ 
for $100$ random values of  $\alpha \in (\alpha_1,\alpha_4)$. 

\begin{table}[htb]
\centering
\begin{tabular}{|c|c|c|}
\hline
$\kappa_2(\f,p)$ & $\frac{\|q-p\|_2}{\|p\|_2}$\\
\hline
$123$ & $0.0436$ \\
\hline \hline
$\kappa_2(\g,p)$ & $\frac{\|r-p\|_2}{\|p\|_2}$ \\
\hline
$1$ &  $0.0221$\\
\hline
\end{tabular}
\end{table}

\noindent As in the example before, the fact that the mean values of $\frac{\|q - p\|_2}{\|p\|_2}$ are smaller than the 
mean values of $\frac{\|r - p\|_2}{\|p\|_2}$ suggests that~$p$ is more stable when 
it is considered as a root of~$\g$ instead of as a root of~$\f$.
\end{example}

\bigbreak\bigbreak



\begin{thebibliography}{99}

\bibitem{ABKR} 
J.\ Abbott, \ A.\ Bigatti, \ M. Kreuzer and L.\ Robbiano 
{\it Computing Ideals of Points},
J. Symb. Comput.\  {\bf 30},  pp 341--356, (2000).

\bibitem{AKR} 
J.\ Abbott,  M.\  Kreuzer and L.\ Robbiano 
{\it Computing zero-dimensional Schemes},
J. Symb. Comput. {\bf 39},  pp 31--49, (2005).

\bibitem{RA} L.\  Robbiano and J.\  Abbott (eds.), 
{\it Approximate Commutative Algebra}, 
Text and Monographs in Symbolic Computation, Springer-Verlag  Wien, 2009

\bibitem{BCM}
D.\ Bini, M.\ Capovani and O.\  Menchi,
{\it Metodi numerici per l'algebra lineare},
Zanichelli 1988.

\bibitem{BPR}
S.\ Basu, R.\ Pollack and M.F.\ Coste-Roy,
{\it Algorithms in Real Algebraic Geometry},
Algorithms and Computation in Mathematics, Vol. 10,
Springer-Verlag 2006.


\bibitem{BM}
B.\ Buchberger, M. M\"oller,
{\it The construction of multivariate polynomials with
preassigned zeros} In J.\ Calmet Editor,
 Proceedings of the European
Computer Algebra Conference ({EUROCAM} '82,
Lecture Notes in Comp.\ Sci.,  {\bf 144}, Springer, pp 24--31,  (1982).

\bibitem{Co}
CoCoATeam,
\cocoa: a system for doing Computations in Commutative Algebra.
Available at http://cocoa.dima.unige.it.

\bibitem{D01}
J.\  D\'egot,
{\it A Condition Number Theorem for Underdetermined Polynomial Systems},
Mathematics of Computation,
{\bf 70}, n. 233, pp 329--335, (2001).

\bibitem{E} D.\  Eisenbud, {\it Commutative algebra with a view toward algebraic geometry}, Graduate Texts in Mathematics, Springer, 1995.

\bibitem{F} C.\ Fassino, {\it Almost Vanishing Polynomials for Sets of Limited Precision Points}, J. Symb. Comput. {\bf 45}, pp 19--37, (2010).

\bibitem{FT} C.\ Fassino, M. Torrente,
{\it Vanishing Polynomials at Sets of Empirical Points}, 
submitted.

\bibitem{GLRR} L.\ Gonzalez, H.\ Lombardi, T.\  Recio and  M.-F. Roy, 
{\it Sturm-Habicht sequence}, In Proceedings of ISSAC'1989,
ACM New York, USA,  pp 136--146 

\bibitem{H96} N.J.\  Higham,
{\it Accuracy and stability of numerical algorithms},
SIAM, 1996.

\bibitem{KPR} M.\ Kreuzer, H.\ Poulisse and L.\ Robbiano, 
{\it From Oil Fields to Hilbert Schemes}, in: 
L.\  Robbiano and J.\  Abbott (eds.), 
{\it Approximate Commutative Algebra}, 
Text and Monographs in Symbolic Computation, Springer-Verlag  Wien,
pp 1--54, (2009).

\bibitem{KR1} M.\ Kreuzer,  L.\ Robbiano, {\it Computational
Commutative Algebra 1}, Springer, Heidelberg 2000.

\bibitem{KR2} M.\ Kreuzer,  L.\ Robbiano, {\it Computational
Commutative Algebra 2}, Springer, Heidelberg 2005.

\bibitem{SS93}
M.\  Shub,  S.\  Smale,
{\it Complexity of Bezout's Theorem I: Geometric Aspects},
Journal of the American Mathematical Society,
{\bf 6} n. 2, pp 459--501, (1993).

\bibitem{SW05}
A.J.\  Sommese,  C.W.\  Wampler,
{\it The numerical solution of systems of polynomials 
arising in engineering and science},
World Scientific, 2005.

\end{thebibliography}
\end{document}